\documentclass[11pt,leqno]{amsart}
\usepackage{amssymb,amsmath,amsthm,amsfonts}
\usepackage {latexsym}
\usepackage{bbm}

\allowdisplaybreaks
\newcommand{\nc}{\newcommand}
\nc{\les}{\lesssim}
\nc{\nit}{\noindent}
\nc{\nn}{\nonumber}
\nc{\D}{\partial}
\nc{\diff}[2]{\frac{d #1}{d #2}}
\nc{\diffn}[3]{\frac{d^{#3} #1}{d {#2}^{#3}}}
\nc{\pdiff}[2]{\frac{\partial #1}{\partial #2}}
\nc{\pdiffn}[3]{\frac{\partial^{#3} #1}{\partial{#2}^{#3}}}
\nc{\abs}[1] {\lvert #1 \rvert}
\nc{\cAc}{{\cal A}_c}
\nc{\cE}{{\cal E}}
\nc{\cF}{{\cal F}}
\nc{\cP}{{\cal P}}
\nc{\cV}{{\cal V}}
\nc{\cQ}{{\cal Q}}
\nc{\cGin}{{\cal G}_{\rm in}}
\nc{\cGout}{{\cal G}_{\rm out}}
\nc{\cO}{{\cal O}}
\nc{\Lav}{{\cal L}_{\rm av}}
\nc{\cL}{{\cal L}}
\nc{\cB}{{\cal B}}
\nc{\cZ}{{\cal Z}}
\nc{\cR}{{\cal R}}
\nc{\cT}{{\cal T}}
\nc{\cY}{{\cal Y}}
\nc{\cX}{{\cal X}}
\nc{\cXT}{{{\cal X}(T)}}
\nc{\cBT}{{{\cal B}(T)}}
\nc{\vD}{{\vec \mathcal{D}}}
\nc{\efield}{\mathcal{E}}
\nc{\vE}{{\vec \efield}}
\nc{\vB}{{\vec \mathcal{B}}}
\nc{\vH}{{\vec \mathcal{H}}}
\nc{\ty}{{\tilde y}}
\nc{\tu}{{\tilde u}}
\nc{\tV}{{\tilde V}}
\nc{\Pc}{{\bf P_c}}
\nc{\bx}{{\bf x}}
\nc{\bX}{{\bf X}}
\nc{\bXYZ}{{\bf XYZ}}
\nc{\bY}{{\bf Y}}
\nc{\bF}{{\bf F}}
\nc{\bS}{{\bf S}}
\nc{\dV}{{\delta V}}
\nc{\dE}{{\delta E}}
\nc{\TT}{{\Theta}}
\nc{\dPsi}{{\delta\Psi}}
\nc{\order}{{\cal O}}
\nc{\Rout}{R_{\rm out}}
\nc{\eplus}{e_+}
\nc{\eminus}{e_-}
\nc{\epm}{e_\pm}
\nc{\eps}{\varepsilon}
\nc{\vnabla}{{\vec\nabla}}
\nc{\G}{\Gamma}
\nc{\w}{\omega}
\nc{\mh}{h}
\nc{\mg}{g}
\nc{\vphi}{\varphi}
\nc{\tlambda}{\tilde\lambda}
\nc{\be}{\begin{equation}}
\nc{\ee}{\end{equation}}
\nc{\ba}{\begin{eqnarray}}
\nc{\ea}{\end{eqnarray}}

\nc{\g}{\gamma}
\nc{\ol}{\overline}

\newtheorem{theorem}{Theorem}[section]
\newtheorem{lemma}[theorem]{Lemma}
\newtheorem{prop}[theorem]{Proposition}
\newtheorem{corollary}[theorem]{Corollary}
\newtheorem{defin}[theorem]{Definition}
\newtheorem{rmk}[theorem]{Remark}

\nc{\pT}{\partial_T}
\nc{\pz}{\partial_z}
\nc{\pt}{\partial_t}
\nc{\la}{\langle}
\nc{\ra}{\rangle}
\nc{\infint}{\int_{-\infty}^{\infty}}
\nc{\halfwidth}{6.5cm}
\nc{\figwidth}{10cm}
\newcommand{\f}{\frac}

\nc{\nlayers}{L} \nc{\nsectors}{M}
\nc{\indicator}{\mathbf{1}}
\nc{\Rhole}{R_{\rm hole}}
\nc{\Rring}{R_{\rm ring}}
\nc{\neff}{n_{\rm eff}}
\nc{\Frem}{F_{\rm rem}}
\nc{\R}{\mathbb R}
\nc{\Z}{\mathbb Z}
\nc{\DD}{\Delta}
\nc{\cD}{\mathcal D}
\nc{\lnorm}{\left\|}
\nc{\rnorm}{\right\|}
\nc{\rnormp}{\right\|_{\ell^{p,\eps}}}
\nc{\rar}{\rightarrow}
\date{\today}
\begin{document}

\begin{abstract}

	We investigate $L^1(\R^n)\to L^\infty(\R^n)$ dispersive estimates for the   Schr\"odinger operator $H=-\Delta+V$ when there is an eigenvalue at zero energy in even dimensions $n\geq 6$. In particular, we show  that if there is an eigenvalue at zero energy then there is a time dependent, rank one operator $F_t$ satisfying $\|F_t\|_{L^1\to L^\infty} \lesssim |t|^{2-\f n2}$ for $|t|>1$  such that
$$\|e^{itH}P_{ac}-F_t\|_{L^1\to L^\infty} \les |t|^{1-\f n2},\,\,\,\,\,\text{ for } |t|>1.$$
With stronger decay conditions on the potential it is possible to generate an operator-valued
expansion for the evolution, taking the form
\begin{align*}
	e^{itH} P_{ac}(H)=|t|^{2-\f n2}A_{-2}+
	|t|^{1-\f n2} A_{-1}+|t|^{-\f n2}A_0,
\end{align*}
with $A_{-2}$ and $A_{-1}$ mapping $L^1(\R^n)$ to $L^\infty(\R^n)$ while $A_0$ maps weighted
$L^1$ spaces to weighted $L^\infty$ spaces.
The leading-order terms $A_{-2}$ and $A_{-1}$ are both finite rank, and vanish
when certain orthogonality conditions
between the potential $V$ and the zero energy eigenfunctions are satisfied.
We show that under the same orthogonality conditions, the  remaining  $|t|^{-\f n2}A_0$ term
also exists as a map from $L^1(\R^n)$ to $L^\infty(\R^n)$, hence $e^{itH}P_{ac}(H)$ satisfies
the same dispersive bounds as the free evolution despite
the eigenvalue at zero.

\end{abstract}

\title[Dispersive Estimates for Schr\"odinger Operators: Even dimensions]{\textit{Dispersive Estimates for higher
dimensional Schr\"odinger Operators with threshold
eigenvalues II: The even dimensional case}}

\author[M.~J. Goldberg, W.~R. Green]{Michael Goldberg and William~R. Green}

\address{Department of Mathematics \\
University of Cincinnati\\
Cincinnati, OH 45221-0025}
\email{Michael.Goldberg@uc.edu}
\address{Department of Mathematics\\
Rose-Hulman Institute of Technology \\
Terre Haute, IN 47803 U.S.A.}
\email{green@rose-hulman.edu}

\thanks{This  work  
was  partially  supported  by  a  grant  from  the  Simons  Foundation (Grant  Number 281057
to  the first author.)
The second author acknowledges the support of an AMS Simons Travel grant and a
Rose-Hulman summer professional development grant.}

\maketitle

\section{Introduction}

In this paper we examine dispersive properties  of the operator $e^{itH}$, where
$H:=-\Delta+V$ with $V$ a real-valued potential on $\R^n$.  The spatial dimension may
be any even number $n \ge 6$, just as Part I of this work,~\cite{GGodd},
considered odd dimensions $n \ge 5$.
This operator
is the propagator of the 
Schr\"odinger equation
\begin{align}\label{Schr}
	iu_t+Hu=0, \qquad u(x,0)=f(x),
\end{align}
as formally,
one can write the solution to \eqref{Schr} as
$u(x,t)=e^{itH}f(x)$. 

When $V=0$, one has the dispersive estimate
$\|e^{itH}\|_{L^1\to L^\infty} \les |t|^{-\f n2}$.  This
can be easily seen by the representation
$$
	e^{-it\Delta}f(x)=\frac{1}{(4\pi i t)^{\f n2}} 
	\int_{\R^n} e^{i|x-y|^2/4t} f(y)\, dy,
$$
which one obtains through elementary properties of the
Fourier transform.
The stability of dispersive estimates under perturbation by a short range potential,
that is for a Schr\"odinger operator of the form $H = -\Delta + V$, where $V$ is 
real-valued and decays at spatial infinity, is a well-studied problem.
Where possible,
the estimate is presented in the form
\begin{equation} \label{eq:dispersive}
\lnorm e^{itH}P_{ac}(H)\rnorm_{L^1(\R^n)\to L^\infty(\R^n)} \lesssim |t|^{-n/2}.
\end{equation}
Projection onto the continuous spectrum is needed as the perturbed Schr\"odinger operator
$H$ may possess  pure point spectrum that experiences no decay at large times.  Under relatively mild assumptions
on the potential 
one has an $L^2$ conservation law for the operator $e^{itH}$.  In addition,  if $|V(x)|\leq C(1+|x|)^{-\beta}$ for some $\beta>1$ and is real-valued,
the spectrum of $H$ is composed of
a finite number of non-positive eigenvalues and purely
absolutely continuous spectrum on $(0,\infty)$, see
\cite{RS}.

The history of this problem is more thoroughly discussed in part I~\cite{GGodd}.
We recall briefly that the first results in the direction of~\eqref{eq:dispersive},
Rauch, Jensen-Kato, Jensen and Murata, \cite{Rauch,JenKat,Jen,Mur,Jen2}, studied mappings
between weighted $L^2(\R^n)$ in place of $L^1(\R^n)$ and $L^\infty(\R^n)$.
Estimates precisely of the form in~\eqref{eq:dispersive} are studied in
\cite{JSS,Wed,RodSch,GS,Sc2,GV,CCV,EG1,BecGol,Gr} by 
a number of authors in various dimensions, and with
different characterizations of the potential $V(x)$
respectively.
The first result on these global, $L^1\to L^\infty$, dispersive estimates
was the work of Journ\'e, Soffer and Sogge \cite{JSS}.
Much of the more recent work has its roots in the work
of Rodnianski-Schlag, \cite{RodSch}.
For a more detailed history, see the survey paper
\cite{Scs}.

Our main concern is the effect of obstructions at zero
energy on the time decay of the evolution.
Jensen and Kato~\cite{JenKat} showed that in three dimensions, if there is a resonance at zero energy then
the propagator $e^{itH}P_{ac}(H)$ (as an operator between polynomially weighted $L^2(\R^3)$ spaces) has leading order decay of $|t|^{-1/2}$
instead of $|t|^{-3/2}$.  In general the same effect occurs if zero is an eigenvalue, even though $P_{ac}(H)$
explicitly projects away from the associated eigenfunction.   Global $L^1\to L^\infty$ dispersive estimates are known in all lower dimensions when zero is
not a regular point of the spectrum, due to Yajima, Erdo\smash{\u{g}}an, Schlag and the authors in various
combinations, see~\cite{GS,ES2,Yaj3,ES,goldE,EG,EGG}. 
The goal of this work is to extend these studies to
all  higher dimension $n>3$.

In dimensions five and higher resonances at zero
do not occur.  In~\cite{Jen} Jensen
obtained leading order decay at the rate $|t|^{2-\frac{n}{2}}$ as an operator on weighted $L^2(\R^n)$ spaces if zero is an eigenvalue.  For  $n\geq 5$, the
subsequent terms of the asymptotic expansion have decay rates $|t|^{1-\frac{n}{2}}$ and $|t|^{-\frac{n}{2}}$
and map between more heavily weighted $L^2(\R^n)$ spaces.  
We are able to recover the same structure
of time decay with respect to mappings from $L^1(\R^n)$ to $L^\infty(\R^n)$, with a finite-rank leading order
term and a remainder that belongs to weighted spaces.
In fact, our results imply Jensen's results on weighted
$L^2(\mathbb R^n)$ spaces with reduced weights.
Perhaps the most surprising result we prove is the full dispersive estimate~\eqref{eq:dispersive} holds
without any spatial weights
if the zero-energy eigenfunctions
satisfy two orthogonality conditions, see Theorem~\ref{thm:main}  part (\ref{thmpart3}) below.

In addition we note that there has been
much study of the wave operators, 
which are defined by strong limits
on $L^2(\R^n)$,
$$
	W_{\pm}=s\mbox{-}\lim_{t\to \pm\infty}
	e^{itH}e^{it\Delta}.
$$
The $L^p$ boundedness of the wave operators, see
\cite{Yaj,FY,JY4}, relates to
dispersive estimates by way of the  
`intertwining property,' which allows us to translate
certain mapping properties of the free propagator to the
perturbed operator,
$$
	f(H)P_{ac}=W_{\pm}f(-\Delta)W_{\pm}^*.
$$
The identity is valid for Borel functions $f$. 
In dimensions $n\geq 5$, boundedness of the 
wave operators on $L^p$ for $\frac{n}{n-2}<p<\frac{n}{2}$
in the presence of an eigenvalue at zero
was established by Yajima~\cite{Yaj} in odd dimensions,
and Finco-Yajima \cite{FY} in even dimensions.
In particular, with $p^\prime$ the conjugate exponent satisfying
$\frac{1}{p}+\frac{1}{p^\prime}=1$, the boundedness of
the wave operators imply the mapping estimate
$$
	\|e^{itH}P_{ac}(H)\|_{L^p\to L^{p\prime}}\les
	|t|^{-\f n2+\frac{n}{p}}.
$$
Roughly speaking, the range of $p$ in 
the wave operator results yield a time decay rate of
$|t|^{-\frac{n}{2}+2+}$.  Similar results in lower dimensions can be found in \cite{Yaj3,JY4}.

The main results in this paper mirror the ones obtained in odd dimensions~\cite{GGodd} and we
will use the same notation and conventions where possible.
Our work here 
is mostly self-contained; we have 
omitted proofs that are proved verbatim, or those that 
require only minor  modifications of those in \cite{GGodd}.
To state our main results,
define a smooth cut-off function $\chi(\lambda)$
with $\chi(\lambda)=1$ if $\lambda<\lambda_1/2$ and
$\chi(\lambda)=0$ if $\lambda>\lambda_1$, for a
sufficiently small $0<\lambda_1\ll 1$.
Further define $\la x\ra :=(1+|x|)$,
then we use the notation for weighted $L^p$ spaces
\begin{equation*}
\lnorm f\rnorm_{L^{p,\sigma}} := \lnorm (1+|x|)^\sigma f\rnorm_p
\end{equation*}
and the abbreviations $a-:=a-\epsilon$ and $a+:=a+\epsilon$ for a small, but
fixed, $\epsilon>0$.
We prove the following low energy bounds.

\begin{theorem}\label{thm:reg}

	Assume that $n \geq 6$ is even, 
	$|V(x)|\les \la x\ra ^{-\beta}$, for some $\beta>n$
	and that zero is  not an eigenvalue of
	$H=-\Delta+V$ on $\R^n$.  Then,
	$$
		\|e^{itH} \chi(H)P_{ac}(H)\|_{L^1\to L^\infty}
		\les |t|^{-\f n2}.
	$$
	
\end{theorem}

\begin{theorem}\label{thm:main}

	Assume that $n \geq 6$ is even, $|V(x)|\les \la x\ra ^{-\beta}$, and that zero is an eigenvalue of
	$H=-\Delta+V$ on $\R^n$.  The low energy Schr\"odinger propagator $e^{itH}\chi(H)P_{ac}(H)$
	possesses the following structure:
	\begin{enumerate}
	\item \label{thmpart1} Suppose that
	there exists $\psi \in {\rm Null}\,H$ such that 
	$\int_{\R^n} V\psi\,dx \not= 0$.  Then there is a rank-one time dependent operator $\|F_t\|_{L^1 \to L^\infty}\les
    |t|^{2-\f n2}$ such that for $|t| > 1,$
	$$
	e^{itH} \chi(H)P_{ac}(H) - 
	F_t= \mathcal E_1(t).
	$$ 
	Where,  $\|\mathcal E_1 \|_{L^1 \to L^\infty}=o(|t|^{2-\f n2})$ if $\beta>n$
	and  $\|\mathcal E_1 \|_{L^1 \to L^\infty}=O(|t|^{1-\f n2})$
	if $\beta>n+4$.
	\item \label{thmpart2} Suppose 
	that $\int_{\R^n} V\psi\, dx = 0$ for each $\psi \in {\rm Null}\,H$ but
           $\int_{\R^n} x_j V \psi\, dx \not= 0$ for some $\psi$ and some $j \in [1, \ldots, n]$.
	 Then there exists a finite-rank time dependent operator $G_t$ satisfying
	$\lnorm G_t\rnorm_{L^1 \to L^\infty} \les |t|^{1-\frac{n}{2}}$ such that for $|t|>1$,
	$$
	e^{itH} \chi(H)P_{ac}(H) - G_t = \mathcal E_2(t).
	$$
	Where,  $\|\mathcal E_2 \|_{L^1 \to L^\infty}=O(|t|^{1-\f n2})$ 
	and
	$ \|\mathcal E_2 \|_{ L^{1,0+} \to L^{\infty,0-}}=o(|t|^{1-\f n2})$
	 if $\beta>n+4$
	and  $\|\mathcal E_2 \|_{L^{1,1} \to L^{\infty,-1}}=O(|t|^{-\f n2})$
	if $\beta>n+8$.
	\item \label{thmpart3} Suppose $\beta>n+8$ and
	that $\int_{R^n} V\psi\, dx = 0$ and $\int_{R^n} x_j V \psi\, dx = 0$
	for all $\psi \in {\rm Null}\,H$ and all $j \in [1,\ldots,n]$.  Then
	$$
	\lnorm e^{itH} \chi(H)P_{ac}(H)\rnorm_{L^1\to L^\infty} \les |t|^{-\frac{n}{2}}
	$$
	\end{enumerate}
	
\end{theorem}

We note that the assumption that
$\int_{\R^n} V\psi\, dx = 0$ for each $\psi \in {\rm Null}\,H$ is equivalent to assuming that the operator
$P_eV1=0$ with $P_e$ the projection onto the zero-energy
eigenspace.  Further, $\int_{R^n} x_j V \psi\, dx = 0$
for each $j=1,2,\dots, n$ 
is equivalent to assuming the operator $P_eVx=0$.

These results are fashioned similarly to the asymptotic expansions in~\cite{Jen}, with particular
emphasis on the behavior of the resolvent of $H$ at low energy.  If one assumes greater decay
of the potential, then it becomes possible to carry out the resolvent expansion to a greater number of
terms, which permits a more detailed description of the time decay of $e^{itH}\chi(H)P_{ac}(H)$.
We note that while $F_t$ and $G_t$ above have a concise construction, expressions for higher order
terms in the expansion are unwieldy enough to discourage writing out an exact formula.

The extension to the main theorem is as follows.

\begin{corollary}\label{cor:ugly}

If $|V(x)|\les \la x\ra^{-n-8-}$, and there is an
eigenvalue of $H$ at zero energy, then 
we have the operator-valued expansion
\begin{align}
	e^{itH} \chi(H)P_{ac}(H)= c|t|^{2-\f n2}P_eV1VP_e +
	|t|^{1-\f n2} A_{-1} + |t|^{-\f n2}A_0(t).
\end{align}
There exist uniform bounds for 
$P_eV1VP_e:L^1\to L^\infty$, 
$A_{-1}:L^{1}\to L^{\infty}$, and $A_0(t):L^{1,2}\to L^{\infty, -2}$.  
The operator $P_eV1VP_e$ is a rank one operator
and $A_{-1}$ is finite rank.
Furthermore, if $P_eV1=0$, then \
$A_0(t):L^{1,1}\to L^{\infty, -1}$. If $P_eV1=0$ and $P_eVx=0$
then $A_{-1}$ vanishes and 
$A_0(t):L^1\to L^\infty$ uniformly in $t$.  

\end{corollary}

We note that this expansion could continue indefinitely
in powers of $|t|^{-\f n2 -k}$, $k\in \mathbb N$.  The
operators would be finite rank 
between successively more heavily
weighted spaces and it would require more decay on the
potential $V$.  We do not pursue this issue.

High energy dispersive bounds in dimension $n \ge 4$
require more assumptions on the smoothness of
the potential, which was shown in the counterexample 
constructed by the first author and Visan
in \cite{GV}.  In contrast the present work is concerned
with the effect of zero energy eigenvalues,
which is strictly a low energy issue.  
Accordingly our theorems
stated above use the low-energy cut-off $\chi(H)$
so that no differentiability on the potential is required.

As in odd dimensions, we note that
the estimates we prove can be combined with
the large energy estimates in, for example, \cite{Yaj, FY} to 
prove analogous statements for the full evolution
$e^{itH}P_{ac}(H)$ without the low-energy cut-off.
The work cited above assumes that
the polynomially weighted Fourier transform of $V$ satisfies
$$
	\mathcal F(\la x\ra^{2\sigma}V)\in L^{n_*}(\R^n)
	\qquad \textrm{ for } \sigma>\frac{1}{n_*}=\frac{n-2}{n-1}.
$$
Roughly speaking, this corresponds to having more than
$\frac{n-3}{2}+\frac{n-3}{n-2}$ derivatives of $V$
in $L^2$.

The statements of our main results are identical to
those given in the companion paper, \cite{GGodd} for
odd dimensions $n\geq 5$.  The analysis for even dimensions
in this paper proceeds along similar lines, but is technically more
challenging.  One reason for this
is the appearance of the logarithms in the 
expansions and the inability to write a closed-form
expression for the resolvents, see \eqref{Yn low}
below.

The limiting resolvent operators are defined as
\begin{align*}
	R_V^{\pm}(\lambda^2)=\lim_{\epsilon \to 0^+}
	(-\Delta+V-(\lambda^2\pm i \epsilon))^{-1}.
\end{align*}
These operators are well-defined on certain weighted
$L^2(\R^n)$ spaces, see \cite{agmon}.  In fact, 
there is a zero energy eigenvalue precisely when this
operator becomes unbounded as $\lambda \to 0$.
While the number of spatial dimensions does not appear
explicitly in the expression above, the behavior of resolvents
for small $\lambda$ is strongly shaped by whether $n$ is 
odd or even.
When odd dimensional resolvents are expanded
in powers of $\lambda$,  
one has the operator-valued expansion
$$
R_V^+(\lambda^2)=\frac{A}{\lambda^2}+\frac{B}{\lambda}+O(1), \qquad 0<\lambda <\lambda_1\ll 1.
$$
In even dimensions one has expansions in terms of
$\lambda^k (\log \lambda)^\ell$. 
For instance, in \cite{EG}
it was shown that in $\R^2$ if there is
a zero energy eigenvalue that one has the operator-valued
expansion (for $0<\lambda<\lambda_1$)
$$
	R_V^{+}(\lambda^2)=\frac{A}{\lambda^2}+
	\frac{B}{\lambda^2(a \log \lambda+ z)} +O(\lambda^{-2}(\log \lambda)^{-2}), \quad  a\in \R\setminus \{0\}, \,\, z \in \mathbb C\setminus \R. 
$$
If, in addition, one  assumes that there are
no zero-energy resonances (solutions to $H\psi=0$ with
$\psi \notin L^2(\R^2)$ but
$\psi \in L^\infty (\R^2)$), one has the expansion
$$
R_V^+(\lambda^2)=\frac{A}{\lambda^2}+
(a \log \lambda+z)B+O((\log\lambda)^{-1}),
$$
with different constants $a,z$ and  
a different operator $B$.  We give only results for
$R_V^+$ since $R_V^-(\lambda^2)=\overline{R_V^+(\lambda^2)}$.
In \cite{EGG} it was shown that the
resolvents in four-spatial dimensions have similar,
though not identical, expansions as those written
above for two dimensions. 
In these lower dimensions it is known that, whether zero is an eigenvalue 
or not, time decay of the Schr\"odinger evolution is faster if there
is not a resonance at zero, see \cite{Mur,ES2,Yaj3,Scs,EG,EG2,EGG} for
example.

As usual (cf. \cite{RodSch,GS,Sc2}), the dispersive estimates follow by considering the operator $e^{itH}\chi(H)P_{ac}(H)$
as an element of the functional calculus of $H$.  Using the Stone formula, and the standard change of
variables $\lambda \mapsto \lambda^2$, we have
\begin{align*}
	e^{itH}\chi(H)P_{ac}(H)f(x)=\frac{1}{2\pi i} 
	\int_0^\infty e^{it\lambda^2} \lambda \chi(\lambda)
	[R_V^+(\lambda^2)-R_V^-(\lambda^2)]f(x)\, 
	d\lambda,
\end{align*}
with the difference of resolvents $R_V^{\pm}(\lambda^2)$ providing the absolutely continuous spectral measure.
For $\lambda > 0$ (and if also at $\lambda = 0$ if zero is a regular point of the spectrum) the resolvents are well-defined on certain
weighted $L^2$ spaces.  The key issue when zero energy is not regular is to control
the singularities in the spectral measure  as $\lambda\to 0$.

Here $R_V^\pm(\lambda^2)$ are operators whose integral
kernel we write as $R_V^\pm(\lambda^2)(x,y)$.  That is,
the action of the operator is defined by
\begin{align*}
	R_V^\pm(\lambda^2)f(x)= \int_{\R^n}R_V^\pm(\lambda^2)(x,y) f(y) \, dy.
\end{align*}

The analysis in this paper focuses on bounding the
oscillatory integral
\begin{align}\label{Stone}
	\int_0^\infty e^{it\lambda^2} \lambda \chi(\lambda)
	[R_V^+(\lambda^2)-R_V^-(\lambda^2)](x,y)\, 
	d\lambda
\end{align}
in terms of $x,y$ and $t$.  A uniform bound of the form
$\sup_{x,y} |\eqref{Stone}|\les |t|^{-\alpha}$ would give
us an estimate on $e^{itH}P_{ac}(H)$ as an operator from
$L^1\to L^\infty$.   We leave open the option of dependence
on $x$ and $y$ to allow for estimates between weighted
$L^1$ and weighted $L^\infty$ spaces.  
That is, an estimate of
the form $|\eqref{Stone}|\les |t|^{-\alpha}\la x\ra^{\sigma'} \la y\ra^{\sigma}$ implies an estimate
for $e^{itH}P_{ac}(H)$
 as an operator from $L^{1,\sigma}$ to $L^{\infty, -\sigma'}$.

The paper is organized as follows.
We begin in Section~\ref{sec:resolv} by developing expansions for the free
resolvent and develop necessary machinery to
understand the spectral measure
$E'(\lambda)=\frac{1}{2\pi i}
[R_V^+(\lambda^2)-R_V^-(\lambda^2)]$.  In
Section~\ref{sec:BS},
we prove dispersive estimates for the finite Born
series series, \eqref{eq:finitebs}, which is the
portion of the low energy evolution
that is unaffected by zero-energy eigenvalues.
Each of these terms experiences time decay of order
$|t|^{-\f n2}$, consistent with the generic dispersive
estimate~\eqref{eq:dispersive}.
  Next, in  Section~\ref{sec:sing}
  we prove dispersive estimates for the tail of the 
  Born series, \eqref{eq:bstail}, which is the portion of the
  evolution that is sensitive to the existence of zero-energy eigenvalues
and to the eigenspace orthogonality conditions specified in Theorem~\ref{thm:main}.
Finally, in Section~\ref{sec:Spec} we provide a characterization
of the spectral subspaces of $L^2$ related to the
zero energy eigenspace and provide technical integral estimates
required to establish the dispersive bounds.

\section{Resolvent Expansions}\label{sec:resolv}

In this section we first develop expansions for the integral kernels of the free
resolvents $R_0^{\pm}(\lambda^2):=(-\Delta-(\lambda^2\pm i0))^{-1}$ to understand the perturbed resolvent operators
$R_V^{\pm}(\lambda^2):=(-\Delta+V-(\lambda^2\pm i0))^{-1}$
with the aim of understanding the spectral measure in
\eqref{Stone}.

In developing these expansions we employ the following
notation used in \cite{GGodd} when considering odd
spatial dimensions.  We write
$$
	f(\lambda)=\widetilde O(g(\lambda))
$$
to indicate that
$$
	\frac{d^j}{d\lambda^j}f(\lambda)=O\bigg(\frac{d^j}{d\lambda^j}g(\lambda)\bigg).
$$
If the relationship holds only for the first $k$ 
derivatives, we use the notation 
$f(\lambda)=\widetilde{O}_k(g(\lambda))$.  With a slight
abuse of notation, we may write $f(\lambda)=\widetilde O(\lambda^k)$ for an integer $k$, to indicate that $\frac{d^j}{d\lambda^j}f(\lambda)=O(\lambda^{k-j})$.
This distinction is particularly important for when
$k\geq 0$ and $j>k$.

Writing the free resolvent in terms of the Hankel functions we have
\begin{align}\label{Hankel}
	R_0(z)(x,y)=\frac{i}{4} \bigg(\frac{z^{1/2}}{2\pi |x-y|}\bigg)^{\frac{n}{2}-1} H_{\frac{n}{2}-1}^{(1)}(z^{1/2}|x-y|).
\end{align}
Here $H_{\frac{n}{2}-1}^{(1)}(\cdot)$ is the Hankel function of the first kind.   When $n$ is even we have
the Hankel function of integer order, which
cannot be expressed in closed form.  This stands in contrast to the odd dimensional free resolvents
which possess a closed form expansion composed of finitely many terms, see for example \cite{Jen}. 
That difference, along
with the appearance of the logarithm in the expansion
\eqref{Yn low}
often makes the even dimensional case more technically
difficult.

We note that
$$
	H_{\frac{n}{2}-1}^{(1)}(z)=J_{\f n2-1}(z)+i Y_{\f n2-1}(z),
$$
where $J_{\f n2-1}$ and $Y_{\f n2-1}$ are the Bessel
functions of integer order.  We note the small $|z|\ll 1$ expansions for the
Bessel functions (c.f. \cite{AS})
\begin{align}
	J_{\f n2-1}(z)&= \bigg(\frac{z}{2}\bigg)^{\f n2-1}
	\sum_{k=0}^\infty\frac{ (-\frac{z^2}{4})^k}
	{k! \Gamma(\f n2+k)}\label{Jn low}\\
	Y_{\f n2-1}(z)&= \frac{-1}{\pi (2z)^{\f n2-1}}\sum_{k=0}^{\f n2-2} \frac{(\f n2-k-2)!}
	{k!}\bigg(\frac{z^2}{4}\bigg)^k+\frac{2}{\pi}
	\log(z/2) J_{\f n2-1}(z)\nn \\
	&-\frac{z^{\f n2-1}}{\pi
	2^{\f n2-1}} \sum_{k=0}^\infty 
	\big\{\psi(k+1)+\psi\big(\f n2+k+2\big)\big\}\frac{(-\frac{1}{4}z^2)}{k! (\f n2-1+k)!}
	\label{Yn low}
\end{align}

In addition, one has the large $|z|\gtrsim 1$ expansion
\begin{align}\label{Jn asymp}
	J_{\f n2 -1}(z)=e^{iz}\omega_+(z)+e^{-iz}\omega_-(z), \qquad
	\omega_{\pm}(z)= \widetilde O(z^{-\f12}).
\end{align}
A similar expansion is valid for $Y_{\f n2 -1}(z)$ with different
functions $\omega_{\pm}(z)$ that satisfy the same bounds.
In fact, such an expansion is valid for any Bessel 
function of integer or half-integer order for
$|z|\gtrsim 1$.

Recall that $R_0^-(\lambda^2)=\overline{R_0^+(\lambda^2)}$.
In particular, using the expansions of the Bessel
functions \eqref{Jn low} and \eqref{Yn low} in \eqref{Hankel}
with $z=\lambda|x-y|$,
we use the following explicit representation for the kernel of the limiting resolvent operators $R_0^\pm(\lambda^2)$ 
(see, {\it e.g.}, \cite{Jen}).  In particular,
\begin{align}\label{even d taylor}
	R_0^+(\lambda^2)(x,y)=\sum_{j=0}^\infty \sum_{k=0}^1
	\lambda^{2j}(\log \lambda)^k G_{j}^k,
\end{align}
which is valid when $\lambda|x-y|\ll 1$ for 
operators $G_j^k$ which are defined by
\begin{align}
	G_j^0&=\left\{ \begin{array}{ll}c_j |x-y|^{2+2j-n}& 
	0\leq j\leq \frac{n}{2}-2\\
	(a_j+ib_j)|x-y|^{2+2j-n}+c_j|x-y|^{2+2j-n}
	\log |x-y| & 
	j\geq \frac{n}{2}-1 \end{array}
	\right.
	\label{Gj0 defn}\\
	G_j^1&=\left\{ \begin{array}{ll}0 & 
	0\leq j\leq \frac{n}{2}-2\\
	b_j|x-y|^{2+2j-n} & 
	j\geq \frac{n}{2}-1 \end{array}
	\right. \label{Gj1 defn}
\end{align}
where $a_j,b_j,c_j \in \R$ and $b_j\neq 0$.

It is worth noting that $G_0^0=(-\Delta)^{-1}$.  To make
the expansions more usable for the purposes of this paper,
when $j\geq \frac{n}{2}-1$, 
we break the operators into
real and imaginary parts.  We define
\begin{align}
	G_j^r&= a_j|x-y|^{2+j-n}+c_j |x-y|^{2+j-n}\log|x-y|,
	\label{Gjreal}\\
	G_j^c&=b_j|x-y|^{2+j-n}.\label{Gjcomp}
\end{align}
We choose to use this representation
since it allows us to
separate operators by the size of
its $\lambda$ dependence as $\lambda\to 0$ and explicitly
identify the imaginary parts of the expansion. 

In addition, the following functions of $\lambda$ occur
naturally in the expansion.
\begin{align}
	g_1^+(\lambda)=\lambda^{n-2}
	(a_1\log \lambda+z_1),\,\,
	g_2^+(\lambda)=\lambda^{n}(a_2\log \lambda+z_2),
	 \,\,
	g_3^+(\lambda)=\lambda^{n+2}(a_3\log \lambda +z_3)
\end{align}
with $a_j\in \R\setminus \{0\}$ and $z_j\in \mathbb C \setminus \R$.  In addition, we have that 
$$
	g_j^-(\lambda)=\overline{g_j^+(\lambda)},
$$
and
\begin{align}\label{eq:g diff}
	g_j^+(\lambda)-g_j^-(\lambda)=2\Im(z_j)\lambda^{n-4+2j},
	\qquad j=1,2,3.
\end{align}

It is worth noting that from the expansions of the
Bessel functions, \eqref{Yn low}, we have
\begin{align}
	g_1^\pm(\lambda)G_{n-2}^c+\lambda^{n-2} G_{n-2}^r
	=\lambda^{n-2}(A_1^\pm+A_2\log(\lambda |x-y|)),
	\label{g1 comb}\\
	g_2^\pm(\lambda) G_n^c
	+\lambda^n G_n^r=\lambda^{n-2}(\lambda|x-y|)^2
	(B_1^\pm+B_2\log (\lambda|x-y|))\label{g2 comb}\\
	g_3^\pm(\lambda) G_{n+2}^c
	+\lambda^{n+2} G_{n+2}^r=\lambda^{n-2}(\lambda|x-y|)^4
	(C_1^\pm+C_2\log (\lambda|x-y|))\label{g3 comb}
\end{align}
for some constants $A_1^\pm,A_2,B_1^\pm, B_2,C_1^\pm,C_2$.  
This follows from \eqref{Hankel} and the expansions 
\eqref{Jn low}, \eqref{Yn low}.  In particular, we note
that the logarithmic factors occur from the
$\log(z/2)J_{\f n2-1}(z)$ terms, which naturally factor
to this form.

Define the function $\log^- (z):= -\chi_{\{0<z<\f12\}}\log (z)$.
Here we note that
\begin{align}\label{eq:log-}
	|(1+\log(\lambda|x-y|))\chi(\lambda|x-y|)
	\chi(\lambda)|\les 1+ |\log \lambda|+\log^- (|x-y|).
\end{align}
This can be seen by considering the cases of 
$|x-y|<1$ and $|x-y|>1$ separately.

\begin{lemma}\label{lem:R0exp}

	For $\lambda\leq \lambda_1$,
	we have the expansion(s) for the free resolvent,
	\begin{align*}
		R_0^{\pm}(\lambda^2)(x,y)=
		G_0^0+\lambda^2 G_1^0+\dots+\lambda^{n-4}
		G_{\frac{n}{2}-2}^0+E_0^{\pm}(\lambda) 
	\end{align*}
	Where
	\begin{align*}
		E_0^{\pm}(\lambda)=(1+\log^-(|x-y|))
		\widetilde O_{\f n2-1}
		(\lambda^{n-2}(1+\log\lambda)).
	\end{align*}
	Further, for $0<\ell<2$,
	\begin{align*}
		E_0^{\pm}(\lambda)&=g_1^\pm(\lambda)G_{n-2}^c
		+\lambda^{n-2} G_{n-2}^r+E_1^\pm(\lambda)
		\qquad \textrm{ with } \qquad E_1^{\pm}(\lambda)=|x-y|^{\ell}
		\widetilde{O}_{\f n2-1}(\lambda^{n-2+\ell}),\\
		E_1^{\pm}(\lambda)&=g_2^\pm(\lambda) G_n^c
		+\lambda^n G_n^r+E_2 ^\pm(\lambda),
		\qquad \textrm{ with } \qquad E_2^{\pm}(\lambda)=|x-y|^{2+\ell}
		\widetilde O_{\frac{n}{2}+1}
		(\lambda^{n+\ell}),\\ 
		E_2^{\pm}(\lambda)&=g_3^\pm(\lambda)G_{n+2}^c
		+\lambda^{n+2}G_{n+2}^r+E_{3}^\pm(\lambda),
		\quad  \textrm{ with } \quad
		E_{3}^\pm(\lambda)=|x-y|^{4+\ell}\widetilde 
		O_{\f n2+3}(\lambda^{n+2+\ell}).
	\end{align*}

\end{lemma}

\begin{proof}

Using the expansion  \eqref{even d taylor}
when $\lambda |x-y|\ll 1$, one has
\begin{align}
	R_0^\pm(\lambda^2)&=G_0^0+\sum_{j=1}^{\frac{n-4}{2}}
	\lambda^{2j}G_{j}^0
	+g_1^\pm(\lambda)G_{n-2}^c
	+\lambda^{n-2} G_{n-2}^r+g_2^\pm(\lambda) G_n^c
	+\lambda^n G_n^r\nn \\
	&+g_3^\pm(\lambda)G_{n+2}^c+\lambda^{n+2}G_{n+2}^r	+\widetilde O(\lambda^{n-2}(\lambda |x-y|)^{6}\log(\lambda|x-y|))\label{eq:R0low}
\end{align}
This can, of course, be truncated eariler.  For 
$E_0^\pm (\lambda)$ we note that for $\lambda|x-y|\ll 1$,
\begin{align*}
	E_0^\pm(\lambda^2)&=-g_1^\pm(\lambda)G_{n-2}^c
	-\lambda^{n-2} G_{n-2}^r+\widetilde O(\lambda^{n-2}(\lambda |x-y|)^{2}\log(\lambda|x-y|))
\end{align*}
For the first two terms, using \eqref{g1 comb}
and \eqref{eq:log-}, we note that
\begin{multline*}
	\lambda^{n-2}G_{n-2}^r+g_1^{\pm}(\lambda)G_{n-2}^c
	=\lambda^{n-2}(A_1^{\pm}+A_2\log(\lambda|x-y|))\\
	=(1+\log^-|x-y|)
	\widetilde O_{\f n2-1}
	(\lambda^{n-2}(1+\log\lambda)).
\end{multline*}

The remaining error bounds for $\lambda |x-y|\ll 1$ are clear from \eqref{eq:R0low}, noting that
$$
	\widetilde O(\lambda^{n-2} (\lambda|x-y|)^2 \log(\lambda|x-y|))=\widetilde O(\lambda^{n-2} (\lambda|x-y|)^\ell))
$$
for any $0\leq \ell <2$.

On the other hand, if $\lambda|x-y|\gtrsim 1$ then
the asymptotic expansion of the Hankel functions
in \eqref{Hankel}, see \eqref{Jn asymp} or
\cite{AS}, yield
\begin{align}\label{R0 high}
	R_0^{\pm}(\lambda^2)
	=e^{\pm i\lambda|x-y|} \frac{\lambda^{\frac{n-2}{2}}}
	{|x-y|^{\frac{n-2}{2}}} \omega_\pm(\lambda|x-y|)
\end{align}
where $\omega_\pm(z)=\widetilde O(z^{-\f12})$.  Here,
differentiation in $\lambda$ in 
is comparable to either division by $\lambda$ or multiplication by $|x-y|$.  So that for $0\leq k\leq
\frac{n}{2}-1$,
\begin{align}\label{R0 high derivs}
	|\partial_\lambda^k R_0^{\pm}(\lambda^2)(x,y)|
	&\les \frac{\lambda^{\frac{n-3}{2}}}
	{|x-y|^{\frac{n-1}{2}}}
	(\lambda^{-k}+|x-y|^k)
	\les\lambda^{\frac{n-3}{2}}|x-y|^{k+\frac{1-n}{2}}
	\les \lambda^{n-2-k} .
\end{align}
Where we used $|x-y|^{-1}\les \lambda$. 
If $k\geq \f n2$, we note that multiplication by
$|x-y|$ dominates division by $\lambda$ in 
\eqref{R0 high derivs}, and we have
\begin{align}
	|\partial_\lambda^k R_0^{\pm}(\lambda^2)(x,y)|
	\les\lambda^{\frac{n-3}{2}}|x-y|^{k+\frac{1-n}{2}}
	\les \lambda^{\frac{n-3}{2}} |x-y|^{\f 12 +k} .
\end{align}
 
The bound for $E_0^\pm(\lambda)$ follows from the bounds
here and the fact that
$$
	E_0^{\pm}(\lambda)=R_0^{\pm}(\lambda^2)
	-G_0^0-\lambda^2 
	G_1^0-\dots-\lambda^{n-4}
	G_{\frac{n}{2}-2}^0.
$$
For these terms, we note that for 
$\lambda|x-y|\gtrsim 1$ and $j\leq \frac{n}{2}-2$ we
have
\begin{align}\label{Gj0 derivs}
	|\partial_\lambda^k\lambda^{2j}G_{j}^0|
	\les \left\{\begin{array}{ll}
	\lambda^{2j-k}|x-y|^{2-n-2j} & k<2j\\
	0 & k\geq 2j\end{array}
	\right. \les \lambda^{n-2-k}.
\end{align}

For the other error terms, we note that
\begin{align*}
	E_1^{\pm}(\lambda)&=E_0^\pm(\lambda)
	+g_1^{\pm}(\lambda)G_{n-2}^c
	+\lambda^{n-2} G_{n-2}^r,\\
	E_2^{\pm}(\lambda)&=E_1^\pm(\lambda)
	+g_2^{\pm}(\lambda) G_n^c
	+\lambda^n G_n^r,\\
	E_3^{\pm}(\lambda)&=E_2^\pm(\lambda)
	+g_3^{\pm}(\lambda) G_{n+2}^c
	+\lambda^{n+2} G_{n+2}^r,
\end{align*}
For these terms, using \eqref{g1 comb}, we note that
when $\lambda|x-y|\gtrsim 1$,
$$
	\lambda^{n-2}G_{n-2}^r+g_1^{\pm}(\lambda)G_{n-2}^c
	=\lambda^{n-2}(A_1^{\pm}+A_2\log(\lambda|x-y|))=
	|x-y|^{0+}\widetilde O(\lambda^{n-2+}).
$$
Similarly, using \eqref{g2 comb},
$$
	\lambda^{n}G_{n}^r+g_2^{\pm}(\lambda)G_{n}^c
	=\lambda^{n}|x-y|^2(B_1^{\pm}+B_2\log(\lambda|x-y|))=
	|x-y|^{2+}\widetilde O(\lambda^{n+}),
$$
and using \eqref{g3 comb}
$$
	\lambda^{n+2}G_{n+2}^r+g_3^{\pm}(\lambda)G_{n+2}^c
	=\lambda^{n+2}|x-y|^4(C_1^{\pm}
	+C_2\log(\lambda|x-y|))=
	|x-y|^{4+}\widetilde O(\lambda^{n+2+}).
$$
Finally, we note that for $\lambda|x-y|\gtrsim 1$,
it is acceptable to multiply upper bounds by powers of $\lambda|x-y|$ .
For $E_j^\pm(\lambda)$, $j=1,2$, we note that
for $\alpha\geq 0$ we have,
\begin{align}
	|\partial_\lambda^k R_0^{\pm}(\lambda^2)(x,y)|
	& \les (\lambda |x-y|)^{\alpha} \left\{\begin{array}{ll} \lambda^{n-2-k} 
	 & k\leq \f n2 -1\\
	\lambda^{\frac{n-3}{2}} |x-y|^{\f 12 +k} & k\geq \f n2
	\end{array}\right.
\end{align}
The bounds then
follow from selecting different values of $\alpha$.

\end{proof}

\begin{corollary}\label{cor:R0exp}

	We have the expansion
	\begin{multline*}
		R_0^{\pm}(\lambda^2)(x,y)=
		G_0^0+\lambda^2 G_1^0+\dots+\lambda^{n-4}
		G_{\frac{n}{2}-2}^0+
		g_1^\pm(\lambda)G_{n-2}^c
		+\lambda^{n-2} G_{n-2}^r\\+
		|x-y|^{\frac{1}{2}+\alpha} \widetilde O_{\f n2}(\lambda^{n-\f32+\alpha})
	\end{multline*}
	for $0\leq \alpha <\f 32$.

\end{corollary}

The hypotheses of the lemma below are not optimal, but
suffice for our purposes.

\begin{lemma}\label{lem:iterated}

	If $|V(x)|\les \la x\ra^{-\frac{n+1}{2}-}$, $\sigma>\f12$ and 
	$\kappa\geq \frac{n-3}{4}$, then
	\begin{align*}
		\| (R_0^\pm(\lambda)^2V)^{\kappa-1}(y,\cdot)R_0(\cdot, x)\|_{L^{2,-\sigma}_y}
		\les \la \lambda \ra^{\kappa\frac{n-3}{2}}.
	\end{align*}
	uniformly in $x$.

\end{lemma}

\begin{proof}

	We note the bound
	$$
		|R_0^\pm(\lambda^2)(x,y)|\les \frac{1}{|x-y|^{n-2}}+
		\frac{\lambda^{\frac{n-3}{2}}}
		{|x-y|^{\frac{n-1}{2}}},
	$$
	which follows from the asymptotic expansion 
	\eqref{R0 high} when $\lambda|x-y|\gtrsim 1$
	and the fact that $|R_0^\pm|\les |G_0^0|\les |x-y|^{2-n}$ for $\lambda|x-y|\ll 1$.  
	The proof follows as in Lemma~2.2 in the odd dimensional
	case, \cite{GGodd}, by repeated use of Lemma~\ref{EG:Lem}.

\end{proof}

We use the symmetric resolvent identity, which 
is valid for $\Im(\lambda)>0$,
\begin{align}\label{symmresid}
	R_V^\pm(\lambda^2)=  R_0^\pm(\lambda^2)-R_0^\pm(\lambda^2)vM^\pm(\lambda)^{-1}vR_0^\pm(\lambda^2),
\end{align}
with $U$ the sign of $V$, $v=|V|^{1/2}$, and $w=Uv$. 
We need to invert
$$
M^{\pm}(\lambda)=U+vR_0^{\pm}(\lambda^2)v
$$
as an operator on $L^2(\R^n)$.

Lemma~\ref{lem:iterated} allows us to make sense of
the symmetric
resolvent identity, provided 
$|V(x)|\les \la x\ra^{-\frac{n-1}{2}-}$,
by iterating the standard resolvent identity
$$
	R_V^{\pm}(\lambda^2)=R_0^{\pm}(\lambda^2)-
	R_0^{\pm}(\lambda^2)VR_V^{\pm}(\lambda^2)
	=R_0^{\pm}(\lambda^2)-
	R_V^{\pm}(\lambda^2)VR_0^{\pm}(\lambda^2)
$$
at least $\frac{n-3}{4}$ times 
on both sides of $M^{\pm}(\lambda)^{-1}$ in 
\eqref{symmresid} to get to a polynomially weighted $L^2$ space, which multiplication by $v$ then maps into $L^2$.

In contrast to the odd dimensional case, \cite{GGodd},
the expansions for the free resolvent 
in Lemma~\ref{lem:R0exp}
are useful for understanding the operators $M^\pm(\lambda)^{-1}$, but
more care is required for the dispersive estimates.
The logarithmic nature of the resolvent causes certain
technical difficulties, see Sections~\ref{sec:BS}
and \ref{sec:sing}.  

Our main tool used to invert 
$M^\pm(\lambda)=U+vR_0^\pm(\lambda^2)v$  
for small $\lambda$  is the
following  lemma (see Lemma 2.1  in \cite{JN}).
\begin{lemma}\label{JNlemma}
Let $A$ be a closed operator on a Hilbert space $\mathcal{H}$ and $S$ a projection. Suppose $A+S$ has a bounded
inverse. Then $A$ has a bounded inverse if and only if
$$
B:=S-S(A+S)^{-1}S
$$
has a bounded inverse in $S\mathcal{H}$, and in this case
$$
A^{-1}=(A+S)^{-1}+(A+S)^{-1}SB^{-1}S(A+S)^{-1}.
$$
\end{lemma}

We use the following terminology.

\begin{defin}
	We say an operator $K:L^2(\R^n)\to L^2(\R^n)$ with kernel
	$K(\cdot,\cdot)$ is absolutely bounded if the operator with kernel
	$|K(\cdot,\cdot)|$ is bounded from $L^2(\R^n)$ to $L^2(\R^n)$.
\end{defin}

We recall the definition of the Hilbert-Schmidt norm
of an operator $K$ with integral kernel
$K(x,y)$ ,
\begin{align*}
	\|K\|_{HS}=\bigg(\iint_{\R^{2n}} |K(x,y)|^2\, dx\, dy
	\bigg)^{\f12}.
\end{align*}
We note that Hilbert-Schmidt and finite rank operators
are immediately absolutely bounded.

\begin{lemma}\label{lem:Mexp}

	Assuming that $v(x)\les \la x\ra^{-\beta}$.
	If $\beta>\frac{n}{2}+\ell$ for any
	$0< \ell<2$, then we have
	\begin{align}
		M^{\pm}(\lambda)
		&=U+vG_0^0v+
		\sum_{j=1}^{\frac{n-4}{2}}
		\lambda^{2j}vG_{j}^0v+g_1^\pm(\lambda)vG_{n-2}^c v
		+\lambda^{n-2}vG_{n-2}^r v+
		M_0^\pm(\lambda),\label{M0exp}
	\end{align}
	Where  the 
	operators $G_{j}^0$, $G_j^r$ and $G_j^c$ are absolutely bounded
	with real-valued kernels.
	Further,
	\begin{align}\label{M0err}
		\sum_{j=0}^{\frac{n}{2}-1} \| 
		\sup_{0<\lambda<\lambda_1} \lambda^{j+2-n-\ell}
		\partial_\lambda^j M_0^{\pm}(\lambda)\|_{HS}
		\les 1.
	\end{align}
	If $\beta>\frac{n}{2}+2+\ell$, for $0< \ell<2$, 
	then
	\begin{align}\label{M0exp2}
		M_0^{\pm}(\lambda)=g_2^\pm(\lambda)vG_{n}^c v
		+\lambda^{n}vG_{n}^r v
		+M_1^{\pm}(\lambda),
	\end{align}
	with
	\begin{align}\label{M1err}
		\sum_{j=0}^{\frac{n}{2}} \| 
		\sup_{0<\lambda<\lambda_1} \lambda^{j-n-\ell}
		\partial_\lambda^j M_1^{\pm}(\lambda)\|_{HS}
		\les 1.
	\end{align}
	If $\beta>\frac{n}{2}+4+\ell$, then for $0< \ell<2$
	\begin{align}\label{M1exp}
		M_1^{\pm}(\lambda)=g_3^\pm(\lambda)vG_{n+2}^c v
		+\lambda^{n+2}vG_{n+2}^r v
		+M_2^{\pm}(\lambda)
	\end{align}
	with
	\begin{align}\label{M2err}
		\sum_{j=0}^{\frac{n}{2}} \| 
		\sup_{0<\lambda<\lambda_1} \lambda^{j-2-n-\ell}
		\partial_\lambda^j M_2^{\pm}(\lambda)\|_{HS}
		\les 1.
	\end{align}

\end{lemma}

\begin{proof}

	The proof follows from the definition of the operators
	$M^{\pm}(\lambda)$ and the expansion for the free
	resolvent in Lemma~\ref{lem:R0exp}.  The bound on the
	error terms follows from the fact 
	that if $k>-\frac{n}{2}$ then
	$\la x\ra^{-\beta}|x-y|^k(1+\log|x-y|) \la y\ra^{-\beta}$ is bounded in Hilbert-Schmidt norm.
	To see this we note that the kernel is bounded by the sum
	$\la x\ra^{-\beta}|x-y|^{k+}\la y\ra^{-\beta}+\la x\ra^{-\beta}|x-y|^{k-} \la y\ra^{-\beta}$
	which are
	Hilbert-Schmidt provided $\beta>\frac{n}{2}+k$.  

\end{proof}

\begin{rmk}

	The error estimates here can be more compactly
	summarized as
	\begin{align*}
		M_0^{\pm}(\lambda)=\widetilde O_{\frac{n}{2}-1}
		(\lambda^{n-2+\ell}), \qquad
		M_1^{\pm}(\lambda)=\widetilde O_{\frac{n}{2}}
		(\lambda^{n+\ell}), \qquad
		M_2^{\pm}(\lambda)=\widetilde O_{\frac{n}{2}}
		(\lambda^{n+2+\ell})
	\end{align*}
	as absolutely bounded operators on $L^2(\R^n)$, 
	for $0<\lambda<\lambda_1$.

\end{rmk}

We note that $U+vG_0^0v$ is not invertible if there
is an eigenvalue at zero, see 
Lemma~\ref{S characterization}.
Define $S_1$ to be the Riesz projection onto the 
kernel of $U+vG_0^0v$ as an operator on $L^2(\R^n)$.  
Then the operator $U+vG_0^0v+S_1$ is invertible on 
$L^2$,  and we may define
\begin{align}
	D_0:=(U+vG_0^0v+S_1)^{-1}.
\end{align}
We note that
$U+vG_0^0v$ is a compact perturbation of the
invertible operator $U$, hence $S_1$ is finite rank by the
Fredholm alternative.
This operator can be seen to be absolutely bounded
exactly as in the odd dimensional case, see 
Lemma~2.7 in \cite{GGodd}.  

\begin{lemma}\label{lem:D0bdd}

	If $v(x)\les \la x\ra^{-\frac{n+1}{2}-}$, then
	the operator $D_0$ is absolutely bounded in $L^2(\mathbb R^n)$.

\end{lemma}

We will apply Lemma~\ref{JNlemma} with $A=M^\pm(\lambda)$ and $S=S_1$,
the Riesz projection onto the kernel of
$U+vG_0^0v$. Thus, we need to show that $M^{\pm}(\lambda)+S_1$
has a bounded inverse in $L^2(\mathbb R^n)$ and
\begin{align}\label{B defn}
  B_{\pm}(\lambda) =S_1-S_1(M^\pm(\lambda)+S_1)^{-1}S_1
\end{align}
has a bounded inverse in $S_1L^2(\mathbb R^n)$.

\begin{lemma}\label{M+S inverse}
	Suppose that zero is not a regular point of the
	spectrum of $H=-\Delta+V$, and let $S_1$ be the
	corresponding Riesz projection on the the zero
	energy eigenspace.  The for sufficiently small
	$\lambda_1>0$, the operators $M^\pm(\lambda)+S_1$
	are invertible for all $0<\lambda<\lambda_1$ as
	bounded operators on $L^2(\R^n)$.  Further,
	for any $0< \ell< 2$, if $\beta>\frac{n}{2}+\ell$
	then we have the following expansions.
	\begin{align*}
		(M^{\pm}(\lambda)+S_1)^{-1}
		&=D_0+\sum_{j=1}^{\frac{n-4}{2}}
		\lambda^{2j}C_{2j}-g_1^\pm(\lambda) 
		D_0vG_{n-2}^c vD_0
		+\lambda^{n-2} C_{n-2}
		+\widetilde M_0^\pm(\lambda)
	\end{align*}
	where $\widetilde M_0^{\pm}(\lambda)$ satisfies the
	same bounds as $M_0^{\pm}(\lambda)$ and the operators
	$C_{k}$ are absolutely bounded on $L^2$ with
	real-valued kernels.
	Further, if $\beta>\frac{n}{2}+2+\ell$ then	
	\begin{align*}
		\widetilde M_0^{\pm}(\lambda)=-
		g_2^\pm(\lambda)D_0vG_n^cvD_0+\lambda^2 g_1^\pm
		(\lambda) C_n^1+
		\lambda^{n}C_{n} +
		\widetilde M_1^\pm (\lambda)
	\end{align*}
	where 
	$C_n^1= D_0vG_{n-2}^cvD_0vG_1^0vD_0
	+D_0vG_1^0vD_0vG_{n-2}^cvD_0$, and
	$\widetilde M_1^{\pm}(\lambda)$ satisfies the
	same bounds as $M_1^{\pm}(\lambda)$.
	Finally, if $\beta>\frac{n}{2}+4+\ell$ then	
	\begin{align*}
		\widetilde M_1^{\pm}(\lambda)=
		-g_3^\pm(\lambda)D_0vG_{n+2}^cvD_0
		+\lambda^2 g_2^\pm
		(\lambda) C_{n+2}^1+\lambda^4 g_1^\pm(\lambda)
		C_{n+2}^2+
		\lambda^{n+2}C_{n+2} +
		\widetilde M_2^\pm (\lambda)
	\end{align*}
	with $C_{n+2}, C_{n+2}^1, C_{n+2}^2$ absolutely
	bounded operators with real-valued kernels and
	$\widetilde M_{2}^\pm(\lambda)$ satisfies the
	same bounds as $M_2^\pm(\lambda)$.

\end{lemma}

\begin{proof}

We use a Neumann series expansion.  We show the case of
$M^+$ and omit the superscript, the `-' case follows
similarly.  Using \eqref{M0exp} we have
\begin{align*}
	(M(\lambda)&+S_1)^{-1}\\
	&=(U+vG_0v+S_1
	+\sum_{j=1}^{\frac{n-4}{2}}
	\lambda^{2j}vG_{j}^0v+g_1(\lambda)vG_{n-2}^c v
			+\lambda^{n-2}vG_{n-2}^r v+
	M_0(\lambda))^{-1}\\
	&=D_0(\mathbbm 1+\sum_{j=1}^{\frac{n-4}{2}}
	\lambda^{2j}vG_{j}^0vD_0+g_1(\lambda)vG_{n-2}^c vD_0
	+\lambda^{n-2}vG_{n-2}^r vD_0+
	M_0(\lambda)D_0)^{-1}\\
	&=D_0-\lambda^2 D_0vG_1^0vD_0
	+\sum_{j=2}^{\frac{n-4}{2}} \lambda^{2j} C_{2j}
	-g_1(\lambda)vG_{n-2}^c vD_0
	-\lambda^{n-2}vG_{n-2}^r vD_0\\
	&\qquad -D_0M_0(\lambda)D_0
	+\lambda^{2}[D_0vG_{1}^0vD_0[g_1(\lambda)vG_{n-2}^c v+\lambda^{n-2}vG_{n-2}^r +M_0(\lambda)]D_0\\
	&\qquad+D_0[g_1(\lambda)vG_{n-2}^c v+\lambda^{n-2}vG_{n-2}^r +M_0(\lambda)]D_0vG_1^0vD_0]
	+\widetilde M_2(\lambda).
\end{align*}
One can find explicitly the operators $C_{k}$ in terms of
$D_0$ and the operators $G_{k}^0$, but this is not worth
the effort.  The operator $C_2=D_0vG_1^0vD_0$ is important due to its
relationship with the projection onto the zero energy
eigenspace, see Lemma~\ref{lem:eproj}.

What is important in our analysis in Section~\ref{sec:sing} are the imaginary parts, that is the
terms that arise with the functions $g_1(\lambda)$,
$g_2(\lambda)$ or $g_3(\lambda)$.  The first of these
occurs from
\begin{align*}
	D_0[g_1(\lambda)vG_{n-2}^c v
	+\lambda^{n-2}vG_{n-2}^r v
	+M_0(\lambda)]D_0
\end{align*}
This provides an most singular term of size $\lambda^{n-2}\log\lambda$ as $\lambda\to 0$.  The
next $\lambda^n \log\lambda$ term arises from
the contribution of the $D_0vM_0(\lambda)vD_0$ 
term
or the `$x^2$' term in the Neumann series,
that is the term with both $G_1^0$ and $G_{n-2}^c$. 
The error bounds follow from the bounds in 
Lemma~\ref{lem:Mexp} and the Neumann series expansion
above.

For the longer expansions, one needs to use more terms
in the Neumann series and take more care with `$x^2$'
and `$x^3$' terms that arise.

\end{proof}

\begin{rmk}\label{rmk:reg}

	We note here that is zero is regular the above
	Lemma suffices to establish the dispersive estimates
	using the techniques in Sections~\ref{sec:BS} and \ref{sec:sing}.
	In this case, $S_1=0$, $D_0=(U+vG_0v)^{-1}$ is still absolutely
	bounded and we have the expansion
	\begin{align*}
		M^{\pm}(\lambda)^{-1}
		&=D_0+\sum_{j=1}^{\frac{n-4}{2}}
				\lambda^{2j}C_{2j}-g_1^\pm(\lambda) 
				D_0vG_{n-2}^c vD_0
				+\lambda^{n-2} C_{n-2}
				+\widetilde M_0^\pm(\lambda),
	\end{align*}
	with $C_{2j}$ real-valued, absolutely bounded operators.

\end{rmk}

Now we turn to the operators $B_{\pm}(\lambda)$ for
use in Lemma~\ref{JNlemma}.  
Recall that
$$
	B_{\pm}(\lambda)=S_1-S_1(M^\pm(\lambda)+S_1)^{-1}S_1,
$$
and that $S_1D_0=D_0S_1=S_1$.  Thus
\begin{align}
	B_{\pm}(\lambda)&=S_1-S_1[D_0
	+\sum_{j=1}^{\frac{n-4}{2}}
	\lambda^{2j}C_{2j}-g_1^\pm(\lambda) 
	D_0vG_{n-2}^c vD_0
	+\lambda^{n-2} C_{n-2}
	+\widetilde M_0^\pm(\lambda)]S_1\nn \\
	&=
	-\sum_{j=1}^{\frac{n-4}{2}}
	\lambda^{2j}S_1C_{2j}S_1+g_1^\pm(\lambda) 
	S_1vG_{n-2}^c vS_1
	-\lambda^{n-2} S_1 C_{n-2}S_1
	-S_1\widetilde M_0^\pm(\lambda)S_1\nn\\
	&=-\lambda^2 S_1vG_1^0vS_1
	-\sum_{j=2}^{\frac{n-4}{2}}
	\lambda^{2j}S_1 C_{2j}S_1+g_1^\pm(\lambda) 
		S_1vG_{n-2}^c vS_1\label{eq:Bexp} \\
	&	\qquad-\lambda^{n-2} S_1 C_{n-2}S_1
	-S_1\widetilde M_0^\pm(\lambda)S_1.	\nn
\end{align}
So that the invertibility of $B_{\pm}(\lambda)$
hinges upon the invertibility of the operator
$S_1vG_1^0vS_1$, which is
established in Lemma~\ref{D1 lemma} below.  Accordingly,
we define  $D_1:=(S_1vG_1^0vS_1)^{-1}$
as an operator on $S_1L^2$.  Noting that
$D_1=S_1D_1S_1$, it is clear that $D_1$ is absolutely
bounded.

\begin{lemma}\label{lem:Binv}

	We have the following expansions, if $\beta>\frac{n}{2}+\ell$ for $0< \ell<2$ then 
	\begin{align*}
		B_{\pm}(\lambda)^{-1}
		&=-\frac{D_1}{\lambda^2}
		+\sum_{j=2}^{\frac{n-4}{2}}
		\lambda^{2j-4}B_{2j}
		+\frac{g_1^\pm(\lambda)} 
		{\lambda^4}D_1vG_{n-2}^c vD_1
		+\lambda^{n-6} B_{n-2}
		+\widetilde B_0^\pm(\lambda)
	\end{align*}
	where $\widetilde B_0^{\pm}(\lambda)$ satisfies the
	same bounds as $\lambda^{-4} M_0^{\pm}(\lambda)$ and the operators
	$B_{k}$ are absolutely bounded on $L^2$ with
	real-valued kernels.
	Further, if $\beta>\frac{n}{2}+2+\ell$ then	
	\begin{align*}
		\widetilde B_0^{\pm}(\lambda)=
		\frac{g_2^\pm(\lambda)}{\lambda^4}D_1vG_n^cvD_1
		+\frac{g_1^\pm(\lambda)}{\lambda^2 } B_n^1
		+\lambda^{n-4}B_{n}+
		\widetilde B_1^\pm (\lambda)
	\end{align*}
	where 
	$B_n^1=D_1vG_{n-2}^cvD_0vG_1^0vD_1
	+D_1vG_1^0vD_0vG_{n-2}^cvD_1+D_1C_4D_0vG_{n-2}^cvD_1
	+D_1vG_{n-2}^cvD_0C_4D_1$, 
	and
	$\widetilde B_1^{\pm}(\lambda)$ satisfies the
	same bounds as $\lambda^{-4} M_1^{\pm}(\lambda)$.
	Finally, if $\beta>\frac{n}{2}+4+\ell$
	\begin{align*}
		\widetilde B_1^{\pm}(\lambda)=
		\frac{g_3^\pm(\lambda)}{\lambda^4}B_{n+2}^1
		+\frac{g_2^\pm(\lambda)}{\lambda^2 } B_{n+2}^2
		+g_1^\pm(\lambda) B_{n+2}^3+
		\lambda^{n-2}B_{n+2}^4+
		\widetilde B_2^\pm (\lambda)
	\end{align*}
	with $B_{n+2}^j$ absolutely bounded operators with
	real-valued kernels, and
	$\widetilde B_2^{\pm}(\lambda)$ satisfies the
	same bounds as $\lambda^{-4} M_2^{\pm}(\lambda)$.

\end{lemma}

\begin{proof}

As usual we consider the `+' case and omit subscripts,
the `-' case follows similarly.  We begin by noting that
\begin{align*}
	B(\lambda)^{-1}&=\Bigl[-\lambda^2 S_1vG_1^0vS_1
	-\sum_{j=2}^{\frac{n-4}{2}}
	\lambda^{2j}S_1 C_{2j}S_1-g_1^\pm(\lambda) 
		S_1vG_{n-2}^c vS_1	+\lambda^{n-2}S_1 C_{n-2}S_1\\
	&\qquad -S_1\widetilde M_0^\pm(\lambda)S_1\Bigr]^{-1}\\
	&=-\frac{D_1}{\lambda^2}\Bigl[\mathbbm 1 +
	\sum_{j=2}^{\frac{n-4}{2}}
	\lambda^{2j-2}S_1 C_{2j}S_1D_1
	-g_1^\pm(\lambda) S_1vG_{n-2}^c vS_1D_1\\
	&\qquad 	+\lambda^{n-2} S_1C_{n-2}S_1D_1
	-\lambda^{-2}S_1\widetilde M_0^\pm(\lambda)
	S_1D_1\Bigr]^{-1}
\end{align*}
where $D_1:=(S_1vG_1^0vS_1)^{-1}$ is an absolutely bounded
operator on $S_1L^2(\R^n)$ by Lemma~\ref{D1 lemma} below.

We again only concern ourselves with explicitly finding
the operators for the first few occurrences of 
the functions $g_1(\lambda)$,
$g_2(\lambda)$ and $g_3(\lambda)$.  The terms that arise
with only powers of the spectral parameter $\lambda$
come with only real-valued, absolutely bounded 
operators which are easier to control.
This again follows by a careful analysis
of the various terms that arise in the Neumann series
expansion.

\end{proof}

\begin{rmk}

	The error estimates here can be more compactly
	summarized as
	\begin{align*}
		\widetilde B_0^{\pm}(\lambda)=\widetilde O_{\frac{n}{2}-1}
		(\lambda^{n-6+\ell}), \qquad
		\widetilde B_1^{\pm}(\lambda)=\widetilde O_{\frac{n}{2}}
		(\lambda^{n-4+\ell}), \qquad
		\widetilde B_2^{\pm}(\lambda)=\widetilde O_{\frac{n}{2}}
		(\lambda^{n-2+\ell})
	\end{align*}
	as absolutely bounded operators on $L^2(\R^n)$, 
	for $0<\lambda<\lambda_1$.  The leading $\lambda^2$
	term in $B_\pm(\lambda)$, \eqref{eq:Bexp}, causes an effective loss of
	four powers of $\lambda$ in the expansion for $B_{\pm}(\lambda)^{-1}$ and hence later for
	$M^\pm(\lambda)^{-1}$ and the perturbed resolvents
	$R_V^\pm(\lambda^2)$.  Heuristically speaking, this
	corresponds to being able to integrate by parts only
	$\frac{n}{2}-2$ times in \eqref{Stone} before the
	integral is too singular as $\lambda \to 0$, which is why a generic
	eigenfunction at zero causes a two power loss of 
	time decay.  This
	loss in the spectral parameter in the expansions,
	necessitates going out to size $\lambda^{n+2+}$ in the
	expansions for $R_0^\pm(\lambda^2)$ to obtain the
	desired $|t|^{-\f n2}$
	time decay in Section~\ref{sec:sing}.

\end{rmk}

To prove parts (\ref{thmpart2}) and (\ref{thmpart3})
of Theorem~\ref{thm:main}, we need the following
corollary.

\begin{corollary}\label{Binv cor}

	Under the hypotheses of Lemma~\ref{lem:Binv}, if
	$P_eV1=0$ then,
	\begin{align*}
		B_{\pm}(\lambda)^{-1}
		&=-\frac{D_1}{\lambda^2}
		+\sum_{j=2}^{\frac{n-4}{2}}
		\lambda^{2j-4}B_{2j}
		+\lambda^{n-6} B_{n-2}
		+\frac{g_2^\pm(\lambda)}{\lambda^4}D_1vG_n^cvD_1
		+\lambda^{n-4}B_{n}+
		\widetilde B_1^\pm (\lambda)
	\end{align*}
	
	If, in addition, $P_{e}Vx=0$ then
	\begin{align*}
		B_{\pm}(\lambda)^{-1}
		&=-\frac{D_1}{\lambda^2}+\sum_{j=1}^{\frac{n-4}{2}}
		\lambda^{2j-4}B_{2j}
		+\lambda^{n-6} B_{n-2}
		+\lambda^{n-4}B_{n}\\
		&+\frac{g_3^\pm(\lambda)}{\lambda^4}B_{n+2}^1
		+\lambda^{n-2}B_{n+2}^4+
		\widetilde B_2^\pm (\lambda)
	\end{align*}

\end{corollary}

\begin{proof}

	We note that $D_1=S_1D_1S_1$, along with the identities
	\begin{align}
		S_1=-wG_0^0vS_1=-S_1vG_0^0w. \label{S1 ident}
	\end{align}
	So that, using  $P_e=G_0^0vD_1vG_0^0$ by \eqref{Pe defn},
	\begin{align}\label{D1Pe ident}
		D_1=S_1D_1S_1=wG_0^0vD_1vG_0^0w=wP_ew.
	\end{align}
	As a consequence, we have 
	\begin{align}\label{D1 useful}
		D_1vG_{n-2}^c=c_{n-2}wP_eV1.
	\end{align}
	The first claim follows clearly from Lemma~\ref{lem:Binv} since the coefficient of
	$\lambda^{n-6}$ is a scalar multiple of the
	operator $P_eV1$. 
	Further,
	\begin{align*}
		c_n^{-1}D_1vG_{n}^cvD_1&=wP_e V [x^2-2x\cdot y+y^2]VP_ew\\
		&=wP_eVx^2 1VP_ew-2wP_eVx\cdot yVP_ew+wP_eV1 y^2 VP_ew.
	\end{align*}
	We see that when $P_eV1=0$ and
	$P_{e}Vx=0$, the operator $D_1vG_{n}^cvD_1=0$.
	We also note that it is now clear that when
	$P_eV1,P_eVx=0$, one has
	$B_n^1=D_1vG_{n-2}^cvD_0vG_1^0vD_1
		+D_1vG_1^0vD_0vG_{n-2}^cvD_1+D_1C_4D_0vG_{n-2}^cvD_1
		+D_1vG_{n-2}^cvD_0C_4D_1=0$ as well.

\end{proof}

	Effectively, all terms that have the function
	$g_1^\pm(\lambda)$ become zero if $P_eV1=0$ and
	all terms with the function $g_2^\pm(\lambda)$ 
	become zero if $P_eVx=0$ as well.

We are now ready to give a full expansion for the
operators $M^{\pm}(\lambda)^{-1}$.  We state several versions of the expansions for
$M^\pm (\lambda)^{-1}$.  These different
expansions allow us to account for cancellation properties
of the eigenfunctions and 
have finer control on the time decay rate of the error terms of the evolution given
in Theorem~\ref{thm:main} at the cost of more decay on
the potential.

\begin{lemma}\label{lem:Minv}

	Assume $|V(x)|\les \la x\ra^{-\beta}$ for some $\beta>n+8$, then
	\begin{multline}
		M^{\pm}(\lambda)^{-1}=-\frac{D_1}{\lambda^2}
		+\sum_{j=0}^{\frac{n-8}{2}} \lambda^{2j}
		M_{2j}+\frac{g_1^\pm (\lambda)}{\lambda^4}M_{n-6}^L+\lambda^{n-6}
		M_{n-6}\\
		+\frac{g_1^\pm(\lambda)}{\lambda^2}
		M_{n-4}^{L1}+\frac{g_2^\pm(\lambda)}{\lambda^4}
		M_{n-4}^{L2}
		+\lambda^{n-4}M_{n-4}\\
		+g_1^\pm(\lambda)M_{n-2}^{L1}+\frac{g_2^\pm(\lambda)}{\lambda^2}M_{n-2}^{L2}+
		\frac{g_3^\pm(\lambda)}{\lambda^4}M_{n-2}^{L3}
		+\lambda^{n-2}M_{n-2}
		+\widetilde{O}_{\f n2}(\lambda^{n-2+})\label{Minv eqn}
	\end{multline}
	for sufficiently small $\lambda$, with all operators
	$M_{k}$ and $M_{k}^{Lj}$ real-valued and absolutely bounded.

\end{lemma}

\begin{proof}

This follows from the expansions in Lemmas~\ref{M+S inverse} and \ref{lem:Binv}, and the inversion lemma,
Lemma~\ref{JNlemma}.

\end{proof}

Later on it will be important to explicitly identify the form of the
operator $M_{n-6}^L$.  We use
Lemma~\ref{M+S inverse} to see that
$$
	(M^\pm(\lambda)+S_1)^{-1}=D_0
	+O(\lambda^2).
$$
Pairing this with the $g_1^\pm(\lambda)$ term in Lemma~\ref{lem:Binv},
the smallest $\lambda$ contribution that is not 
strictly real-valued is
\begin{align*}
	\frac{g_1^\pm(\lambda)}{\lambda^4}D_0
	D_1vG_{n-2}^c
	vD_1	D_0.
\end{align*}
Since $D_1D_0=D_0D_1=D_1$, we have
\begin{align}\label{Mn-6L}
	M_{n-6}^L=D_1vG_{n-2}^cvD_1=wP_eV1VP_ew.
\end{align}

The expansion \eqref{Minv eqn}
can be truncated to require less decay on
the potential by using less of the expansions in
Lemmas~\ref{lem:R0exp} and \ref{lem:Binv}.  Specifically,
stopping with the error terms $\widetilde M_0^\pm(\lambda)$ and
$\widetilde B_0^\pm (\lambda)$ respectively with $\ell=0+$.

\begin{corollary}\label{cor:Minv no cancel}

	Assume $|V(x)|\les \la x\ra^{-n-}$, then
	\begin{align}
		M^{\pm}(\lambda)^{-1}=-\frac{D_1}{\lambda^2}
		+\sum_{j=0}^{\frac{n-8}{2}} \lambda^{2j}
		M_{2j}+\frac{g_1^\pm (\lambda)}{\lambda^4}M_{n-6}^L+\lambda^{n-6}
		M_{n-6}
		+\widetilde{O}_{\f n2-1}
		(\lambda^{n-6+}).\label{Minv eqn1}
	\end{align}
	If $|V(x)|\les \la x\ra^{-n-4-}$, then
	\begin{align}
		M^{\pm}(\lambda)^{-1}&=-\frac{D_1}{\lambda^2}
		+\sum_{j=0}^{\frac{n-8}{2}} \lambda^{2j}
		M_{2j}+\frac{g_1^\pm (\lambda)}{\lambda^4}M_{n-6}^L+\lambda^{n-6}
		M_{n-6}\nn\\
		&+\frac{g_1^\pm(\lambda)}{\lambda^2}
		M_{n-4}^{L1}+\frac{g_2^\pm(\lambda)}{\lambda^4}
		M_{n-4}^{L2}
		+\lambda^{n-4}M_{n-4}
		+\widetilde{O}_{\f n2}
		(\lambda^{n-4+}).\label{Minv eqn2}
	\end{align}
	with the operators
	$M_{2j}$ and $M_{2j}^{Lk}$ all real-valued and absolutely bounded.

\end{corollary}


The lemma can also be modified to better account for 
cancellation properties of the projection onto the
zero-energy eigenspace.

\begin{corollary}\label{cor:Minv}

	Under the hypotheses of Lemma~\ref{lem:Minv}, 
	if $P_eV1=0$ and
	 $|V(x)|\les \la x\ra^{-n-4-}$, then
	\begin{align}
		M^{\pm}(\lambda)^{-1}&=-\frac{D_1}{\lambda^2}
		+\sum_{j=0}^{\frac{n-8}{2}} \lambda^{2j}
		M_{2j}+\lambda^{n-6}
		M_{n-6}+\frac{g_2^\pm(\lambda)}{\lambda^4}
		M_{n-4}^{L2}\nn\\
		&+\lambda^{n-4}M_{n-4}
		+\widetilde{O}_{\f n2}(\lambda^{n-4+})
		\label{Minv PV101}
	\end{align}
	If $|V(x)|\les \la x\ra^{-n-8-}$, then
	\begin{align}
		M^{\pm}(\lambda)^{-1}&=-\frac{D_1}{\lambda^2}
		+\sum_{j=0}^{\frac{n-8}{2}} \lambda^{2j}
		M_{2j}+\lambda^{n-6}
		M_{n-6}+\frac{g_2^\pm(\lambda)}{\lambda^4}
		M_{n-4}^{L2}\nn\\
		&+\lambda^{n-4}M_{n-4}
		+\frac{g_2^\pm(\lambda)}{\lambda^2}M_{n-2}^{L2}+
		\frac{g_3^\pm(\lambda)}{\lambda^4}M_{n-2}^{L3}
		+\lambda^{n-2}M_{n-2}
		+\widetilde{O}_{\f n2}(\lambda^{n-2+})
		\label{Minv PV102}
	\end{align}	
	If in addition, $P_eVx=0$, and $|V(x)|\les \la x\ra^{-n-8-}$,
	 then
	\begin{align}
		M^{\pm}(\lambda)^{-1}&=-\frac{D_1}{\lambda^2}
		+\sum_{j=0}^{\frac{n-8}{2}} \lambda^{2j}
		M_{2j}+\lambda^{n-6}
		M_{n-6}
		+\lambda^{n-4}M_{n-4}\nn \\
		&+\frac{g_3^\pm(\lambda)}{\lambda^4}M_{n-2}^{L3}
		+\lambda^{n-2}M_{n-2}
		+\widetilde{O}_{\f n2}(\lambda^{n-2+})
	\end{align}

\end{corollary}

\begin{proof}

	The proof follows as in the proof of 
	Lemma~\ref{lem:Minv} using Corollary~\ref{Binv cor}
	in place of Lemma~\ref{lem:Binv}.
	

\end{proof}

\section{The finite Born series terms}\label{sec:BS}

In this section we estimate the contribution of the
finite Born series, \eqref{eq:finitebs} showing that it
can be bounded by $|t|^{-\f n2}$ uniformly in $x$ and $y$.
These terms in the expansion of the spectral
measure contain only the free resolvent $R_0^\pm(\lambda^2)$
and therefore are not sensitive to the existence of 
zero energy eigenvalues or their cancellation properties.
In even dimensions the lack of a closed form representation for 
$R_0^\pm(\lambda^2)$ causes much more
technical difficulties in these calculations as compared to the
corresponding section in~\cite{GGodd}.
Many of the techniques we develop here to
overcome these difficulties are vital in controlling
the more singular terms considered in Section~\ref{sec:sing}.

Iterating the standard resolvent identity
$$
	R_V^\pm(\lambda^2)=R_0^\pm(\lambda^2)
	-R_0^\pm(\lambda^2)VR_V^\pm(\lambda^2)=
	R_0^\pm(\lambda^2)
	-R_V^\pm(\lambda^2)VR_0^\pm(\lambda^2),
$$
we form the following identity.
\begin{align}
	R_V^\pm(\lambda^2)&=\sum_{k=0}^{2m+1}(-1)^k
	R_0^\pm(\lambda^2)[VR_0^\pm(\lambda^2)]^k\label{eq:finitebs}\\
	&+[R_0^\pm(\lambda^2)V]^{m}R_0^\pm(\lambda^2)v
	M^\pm(\lambda)^{-1}vR_0^\pm(\lambda^2)
	[VR_0^\pm(\lambda^2)]^{m}.\label{eq:bstail}
\end{align}
In light of Lemma~\ref{lem:iterated} the identity 
holds for $m+1\geq \frac{n-3}{4}$ and 
$|V(x)|\les \la x\ra^{-\frac{n+1}{2}-}$ as an identity
from $L^{2,\f12+}\to L^{2,-\f12-}$, as in the 
limiting absorption principle.
 

\begin{prop}\label{bsprop}

	The contribution of \eqref{eq:finitebs} to
	\eqref{Stone} is bounded by $|t|^{-\f n2}$
	uniformly in $x$ and $y$.  That is,
		$$
			\sup_{x,y\in \R^n}\bigg|
			\int_0^\infty e^{it\lambda^2} \lambda \chi(\lambda)\bigg[\sum_{k=0}^{2m+1}(-1)^k\big\{
			R_0^+(VR_0^+)^k
			-R_0^-(VR_0^-)^k\big\}\bigg](\lambda^2)(x,y)\, 
			d\lambda\bigg| \les |t|^{-\f n2}.
		$$

\end{prop}  

We prove this claim with  series of Lemmas.
The following corollary to Lemma~\ref{lem:R0exp} 
is useful.

\begin{lemma}\label{iteratedBSexp}

	We have the expansion
	\begin{align}
		(R_0^\pm(\lambda^2)V)^k R_0^\pm(\lambda^2)(x,y)&=
		K_0+\lambda^2 K_2+\dots +\lambda^{n-4} K_{n-4}
		+\widetilde E^\pm_0(\lambda)(x,y),\nn
	\end{align}
	here the operators $K_j$ have real-valued kernels.
	Furthermore, the error term $\widetilde{E}^\pm_0(\lambda)$ satisfies
	$$
		\widetilde E^\pm_0(\lambda)(x,y)=
		(1+\log^-|x-\cdot|+\log^-|\cdot-y|)
		\widetilde O_{\f n2-1}(\lambda^{n-2-}).
	$$
	Furthermore, if one wishes to have $\f n2$ derivatives, the extended expansion
	$$
		\widetilde E^\pm_0(\lambda)(x,y)
		=g_1^\pm(\lambda) K_{n-2}^c
				+\lambda^{n-2}K_{n-2}^r
		+\widetilde E^\pm_1(\lambda)(x,y),
	$$
	satisfies the bound
	$$
		\widetilde E^\pm_1(\lambda)(x,y)=
		\la x\ra^{\f12} \la y \ra^{\f12}
		\widetilde O_{\f n2}(\lambda^{n-\f32}).
	$$

\end{lemma}

\begin{proof}

	This follows from the expansions for 
	$R_0^\pm(\lambda^2)$ in Lemma~\ref{lem:R0exp}  for
	$\widetilde E_0^\pm(\lambda)(x,y)$
	or Corollary~\ref{cor:R0exp} for 
	$\widetilde E_1^\pm(\lambda)(x,y)$.

	For the iterated resolvents, the desired bounds come
	from simply multiplying out the terms.  It is
	easy to see that
	$$
		K_0= (G_0^0 V)^k G_0^0
	$$
	and
	$$
		K_2=\sum_{j=0}^{k} (G_0^0 V)^j G_1^0 (VG_0^0)^{k-j}
	$$
	one can obtain similar expressions for the other
	operators, but they are not needed.

\end{proof}

\begin{rmk} \label{rmk:33}
The spatially weighted bound $|\partial_\lambda^{\f n2}\widetilde{E}_1^\pm(\lambda)(x,y)|
\les \la x \ra^{\f 12}\lambda^{\frac{n-3}{2}}$ is only needed if all $\f n2$ derivatives act on the
leading resolvent, $R_0^\pm(\lambda^2)(x,z_1)$, in the product.   Similarly, the upper bound $\la y \ra^{\f 12}\lambda^{\frac{n-3}{2}}$ is
only needed if all derivatives act on the lagging resolvent, $R_0^\pm(\lambda^2)(z_k,y)$, in the product.  All other expressions that
arise would be consistent with $\widetilde{E}_1^\pm(\lambda)$ belonging to the class 
$\widetilde{O}_{\f n2}(\lambda^{n-2-})$.
\end{rmk}

The desired time decay follows from taking the difference
and noting that
\begin{multline*}
	[(R_0^+(\lambda^2)V)^k R_0^+(\lambda^2)-
	(R_0^-(\lambda^2)V)^k R_0^-(\lambda^2)](x,y)\\=
	[g_1^+(\lambda)-g_1^-(\lambda)] K_{n-2}^c
	+\widetilde E^+_1(\lambda)(x,y)-\widetilde E^-_1(\lambda)
	(x,y)\\
	=c_1 \lambda^{n-2}K_{n-2}^c+
	\la x\ra^{\f12} \la y \ra^{\f12}
	\widetilde O_{\f n2}(\lambda^{n-\f32}).
\end{multline*}
The first term contributes $|t|^{-\f n2}$
by Lemma~\ref{lem:IBP} as an operator
from $L^1\to L^\infty$, whereas the second term can
be bounded by $|t|^{-\f n2}$ (from Corollary~\ref{cor:fauxIBP2}), but maps
$L^{1,\f12}\to L^{\infty,-\f12}$.  This method fails
to obtain an unweighted $L^1\to L^\infty$
only when all the $\lambda$ derivatives act on
either a leading or lagging free resolvent.  In the
following Lemmas, we show how the unweighted bound
can be achieved.

The following variation of stationary phase from \cite{Sc2} will be useful in the analysis.

\begin{lemma}\label{stat phase}

	Let $\phi'(\lambda_0)=0$ and $1\leq \phi'' \leq C$.  Then,
  	\begin{multline*}
    		\bigg| \int_{-\infty}^{\infty} e^{it\phi(\lambda)} a(\lambda)\, d\lambda \bigg|
    		\lesssim \int_{|\lambda-\lambda_0|<|t|^{-\frac{1}{2}}} |a(\lambda)|\, d\lambda\\
    		+|t|^{-1} \int_{|\lambda-\lambda_0|>|t|^{-\frac{1}{2}}} \bigg( \frac{|a(\lambda)|}{|\lambda-\lambda_0|^2}+
    		\frac{|a'(\lambda)|}{|\lambda-\lambda_0|}\bigg)\, d\lambda.
  	\end{multline*}

\end{lemma}

	Rather than use the expansions of Lemma~\ref{lem:R0exp}, we need to utilize finer
	cancellation properties of the free resolvents than
	can be captured in these expansions.

	We note that
	by \eqref{Hankel} and the definition of the Hankel
	functions, we have
	\begin{align}\label{Jn canc}
		[R_0^+-R_0^-](\lambda^2)(x,y)=
		\frac{i}{2}\bigg(\frac{\lambda}{2\pi |x-y|}
		\bigg)^{\frac{n}{2}-1} J_{\f n2 -1}(\lambda |x-y|)
	\end{align}
	Noting \eqref{Jn low}, for $\lambda|x-y|\ll 1$, 
	we have
	\begin{align}
		[R_0^+-R_0^-](\lambda^2)(x,y)&=
		\frac{i}{2}\bigg(\frac{\lambda}{2\pi |x-y|}
		\bigg)^{\frac{n}{2}-1} \bigg(\frac{\lambda|x-y|}{2}\bigg)^{\f n2 -1}
		\sum_{k=0}^\infty c_k (\lambda |x-y|)^{2k}\nn\\
		&=\lambda^{n-2} G_{n-2}^c+\widetilde O(\lambda^{n-2} (\lambda|x-y|)^{\epsilon}),
		\qquad 0\leq \epsilon <2. \label{R0diffeps}
	\end{align}
	In particular, we note that there are no logarithms
	in this expansion. 
	On the other hand, if $\lambda|x-y|\gtrsim 1$, using
	\eqref{Jn asymp}, we have
	\begin{align}\label{eq:freediffhigh}
		[R_0^+-R_0^-](\lambda^2)(x,y)&=
		\frac{\lambda^{\f n2 -1}}{|x-y|^{\f n2 -1}}
			\bigg(e^{i\lambda|x-y|}\omega_+(\lambda|x-y|)
		+e^{-i\lambda|x-y|}\omega_-(\lambda|x-y|)
		\bigg).
	\end{align}

\begin{lemma}\label{lem:Jn est}

	We have the expansion	
	$$
		[R_0^+-R_0^-](\lambda^2)(x,y)
		=\widetilde O_{\f n2-1}(\lambda^{n-2})
	$$

\end{lemma}

\begin{proof}

	This follows from \eqref{R0diffeps} with $\epsilon=0$,
	\eqref{eq:freediffhigh} and \eqref{R0 high derivs}
	in the proof of Lemma~\ref{lem:R0exp}.

\end{proof}

To best utilize certain cancellations between the
difference of the iterated resolvents, we note the
following algebraic fact,
\begin{align}\label{alg fact}
	\prod_{k=0}^MA_k^+-\prod_{k=0}^M A_k^-
	=\sum_{\ell=0}^M \bigg(\prod_{k=0}^{\ell-1}A_k^-\bigg)
	\big(A_\ell^+-A_\ell^-\big)\bigg(
	\prod_{k=\ell+1}^M A_k^+\bigg).
\end{align}

When applied to the summand in Proposition~\ref{bsprop} it yields
operators of the form $(R_0^- V)^j(R_0^+ - R_0^-)(VR_0^+)^\ell$,
with $j + \ell = k$.  We separate them further into cases where the
difference $R_0^+ - R_0^-$ occurs on the leading resolvent of the product
(i.e. $j = 0$), the lagging resolvent ($\ell = 0$), or a generic position in the interior.

The first case of the difference occuring on a leading or lagging resolvent is the most delicate.  If the difference
acts on an inner resolvent, we obtain an extra
$\lambda^{n-2}$ smallness from Lemma~\ref{lem:Jn est}.
This extra smallness, along with using some recurrence
relationships for the free resolvents in Lemma~\ref{recurrence} allow us to avoid using expansions
for the leading and lagging resolvents to more easily
obtain the time decay.  This is done in detail in Lemma~\ref{lem:BS2} and follows quickly from the
arguments in the more delicate case considered in 
Lemma~\ref{lem:BS1}.  

With respect to avoiding spatial weights Remark~\ref{rmk:33}
explains that we need only consider when the first
$\frac{n}{2}-1$ derivatives when integrating by parts
act on a leading (respectively lagging) resolvent.  Instead
of integrating by parts the final time, we use a modification of stationary phase from Lemma~\ref{stat phase} to attain the time decay and avoid the spatial weights.

\begin{lemma}\label{lem:BS1}

	If $|V(x)|\les \la x\ra^{-\frac{n+2}{2}-}$,
	we have the bound
	\begin{align*}
		\sup_{x,y\in \R^n} 
		\bigg|\int_0^\infty e^{it\lambda^2}
		\lambda \chi(\lambda) \big\{
		[R_0^+-R_0^-](\lambda^2)
		(VR_0^+)^k(\lambda^2)\big\}(x, y)\, d\lambda\bigg|
		\les |t|^{-\f n2}
	\end{align*}

\end{lemma}

\begin{proof}

	By Lemma~\ref{iteratedBSexp}, Remark~\ref{rmk:33} and the discussion
	following it, we need only consider the contribution
	when, upon integrating by parts, all of the derivatives act on the leading or lagging free resolvent.  In the proof we consider when all derivatives act on the leading difference of free
	resolvents, which we regard as the most delicate
	case.  As the remaining operator $(VR_0^+)^k\chi(\lambda)$ is left undisturbed,
	it suffices to note that it has a bounded kernel, uniformly in $\lambda$.
	The case where
	all derivatives act on the lagging free resolvent is somewhat delicate as well; 
	this term fits best in the framework 
	of Lemma~\ref{lem:BS2} below.
	
	For all other placement of derivatives, we note that if any derivatives
	act on `inner resolvents' or the 
	cut-off, an error bound with 
	polynomial weights suffices as growth in these
	variables is controlled by the decay of the 
	surrounding potentials.  Meanwhile, at most
	$\f n2-1$ derivatives would act on a leading or lagging
	resolvent so that they too can be bounded without
	weights.  

	Unlike in the odd dimensional case,
	one must consider the small 
	and large $\lambda |x-z_1|$ regimes separately.
	Using \eqref{R0diffeps},
	the small $\lambda|x-z_1|$ regime requires bounding
	\begin{align}\label{eq:Bornlowbad}
		\bigg|\int_0^\infty e^{it\lambda^2}
		\lambda \chi(\lambda) [\lambda^{n-2}+
		\chi(\lambda|x-z_1|)
		|x-z_1|^{\epsilon}\widetilde O(\lambda^{n-2+\epsilon})] \, d\lambda\bigg|	
		\les |t|^{-\f n2}.
	\end{align}
	The contribution of the first term follows from
	Lemma~\ref{lem:IBP}.  The second term is bounded
	by using a slight modification of Lemma~\ref{lem:fauxIBP}.  In particular, we can
	safely integrate by parts $\f n2-1$ times without
	boundary terms to get
	$$
		|t|^{1-\f n2}\int_0^{\infty}e^{it\lambda^2}
		\chi(\lambda) \chi(\lambda|x-z_1|)
		|x-z_1|^{\epsilon}\widetilde O(\lambda^{1+\epsilon})
		\, d\lambda.
	$$
	The integral
	can be broken up into two pieces, on $0<\lambda<|t|^{-\f12}$ we take $\epsilon=0$ and integrate to gain the extra power of $|t|^{-1}$.
	On $|t|^{-\f12}<\lambda $, we wish to gain another
	$|t|^{-1}$.  First, if no derivatives act on the
	cut-off $\chi(\lambda |x-z_1|)$ we see that
	\begin{align*}
		\int_{|t|^{-\f12}}^\infty e^{it\lambda^2} 
		\chi(\lambda) \chi(\lambda|x-z_1|)&
		|x-z_1|^{\epsilon}\widetilde O(\lambda^{1+\epsilon})
		\, d\lambda \les \frac{|x-z_1|^{\epsilon}\lambda^{\epsilon}}{|t|}
		\bigg|_{\lambda=|t|^{-\f12}}\\
		&+\frac{1}{|t|}\int_{|t|^{-\f12}}^\infty  e^{it\lambda^2} \chi(\lambda) \chi(\lambda|x-z_1|)
		|x-z_1|^{\epsilon}\widetilde O(\lambda^{\epsilon-1})\, d\lambda
	\end{align*}
	Integrating by parts again on the second term and
	taking $\epsilon>0$ small enough 
	(say $\epsilon=\f12$), we
	can bound with
	\begin{align}
		\frac{|x-z_1|^{\epsilon}\lambda^{\epsilon}}{|t|}
		\bigg|_{\lambda=|t|^{-\f12}}&+ \frac{|x-z_1|^{\epsilon}\lambda^{\epsilon-2}}{|t|^2}
		\bigg|_{\lambda=|t|^{-\f12}}+\frac{1}{|t|^2}
		\int_{|t|^{-\f12}}^\infty 
		\chi(\lambda) \chi(\lambda|x-z_1|)
		|x-z_1|^\epsilon
		\widetilde O(\lambda^{\epsilon-3})\, d\lambda\nn\\
		&\les \frac{|x-z_1|^{\epsilon}\lambda^{\epsilon}}{|t|}
		\bigg|_{\lambda=|t|^{-\f12}}+ \frac{|x-z_1|^{\epsilon}\lambda^{\epsilon-2}}{|t|^2}
		\bigg|_{\lambda=|t|^{-\f12}}+\frac{1}{|t|^2}\nn\\
		&\les |x-z_1|^{\epsilon} |t|^{-1-\f\epsilon 2}
		+|t|^{-2}\les
		|t|^{-1}\label{eq:thingy}
	\end{align}
	The last inequality follows from $1\gtrsim 
	\lambda|x-z_1|>
	|t|^{-\f12}|x-z_1|$, which implies $|x-z_1|\les |t|^{\f12}$.  We also used that $\chi'(\lambda)$
	is supported on $\lambda \approx 1$, so the bound
	$|\chi'(\lambda)|\les \lambda^{-1}$ is true.
	
	If, when integrating by parts, the derivative acts
	on the cut-off $\chi(\lambda|x-z_1|)$ we can bound by
	\begin{align*}
		\frac{|x-z_1|^{\epsilon}\lambda^\epsilon}{t}
		\bigg|_{\lambda=t^{-\f 12}}&+\frac{1}{t} 
		\int_{t^{-\f 12}}^\infty |x-z_1|^{1+\epsilon}
		\chi'(\lambda |x-z_1|) \lambda^{\epsilon}\,
		d\lambda\\
		&\les t^{-1}+\frac{|x-z_1|^{1+\epsilon}}{t}
		\int_{\lambda \sim |x-z_1|^{-1}}\lambda^\epsilon
		\, d\lambda \les t^{-1}.
	\end{align*}
	Here the boundary term is bounded by $|t|^{-1}$ as
	before, and the support of $\chi'(\lambda|x-z_1|)$
	implies that $\lambda \approx |x-z_1|^{-1}$.  A similar
	argument covers the case when the derivative acts
	on $\chi(\lambda|x-z_1|)$ in the second integration
	by parts in \eqref{eq:thingy}.
	
	For the $\lambda|x-z_1|\gtrsim 1$ regime, we still
	consider only the most
	delicate term arises when all the derivatives act
	on the leading difference of free resolvents.
	Without loss of generality, we take $t>0$.
	We note that the most difficult term from the
	contribution of \eqref{eq:freediffhigh}
	occurs with the negative phase.  Here,
	one has to bound
	\begin{align*}
		\int_0^\infty e^{it\lambda^2}\lambda \chi(\lambda)
		e^{-i\lambda|x-z_1|}\frac{\lambda^{\f n2 -1}}{|x-z_1|^{\f n2 -1}} \omega_-(\lambda|x-z_1|)
		\, d\lambda
	\end{align*}
	We note that the $\lambda$ smallness and the 
	support of the cut-off $\chi(\lambda)$ allow us to
	integrate by parts $\f n2-1$ times without boundary
	terms, noting
	the second to last bound in 
	\eqref{R0 high derivs} with $k=\f n2 -1$, we need
	to control
	\begin{align}\label{a osc int}
		\frac{1}{|t|^{\f n2-1}} \int_0^\infty e^{it\lambda^2-i\lambda|x-z_1|} \chi(\lambda)
		a(\lambda) \, d\lambda
	\end{align}
	where by \eqref{Jn asymp},
	\begin{align}\label{eq:a bds}
		|a(\lambda)|\les \frac{\lambda^{\f 12}}
		{|x-z_1|^{\f12}}, \qquad
		|a'(\lambda)|\les \frac{1}{\lambda^{\f12}|x-z_1|^{\f 12}}.
	\end{align}
	The stationary point of the phase occurs at
	$\lambda_0=\frac{|x-z_1|}{2t}$.
	By Lemma~\ref{stat phase}, we need to bound
	three integrals,
	\begin{multline}\label{eq:ABC}
		\int_{|\lambda-\lambda_0|<t^{-\f12}}|a(\lambda)|\,
		d\lambda +|t|^{-1} \int_{|\lambda-\lambda_0|>t^{-\f12}}
		\bigg(
		\frac{|a(\lambda)|}{|\lambda-\lambda_0|^2}+
		\frac{|a'(\lambda)|}{|\lambda-\lambda_0|}\bigg)\, 
		d\lambda\\ :=A+|t|^{-1}(B+C).
	\end{multline}
	We begin by showing that $A\les |t|^{-1}$.  There
	are two cases to consider.  First, if $\lambda_0\gtrsim t^{-\f12}$, we have $\lambda \les \lambda_0$, so that
	$$
		A\les \int_{|\lambda-\lambda_0|<t^{-\f12}}
		\frac{\lambda_0^{\f12}}{|x-z_1|^{\f12}}\, 
		d\lambda \les t^{-\f12} \lambda_0^{\f 12}|x-z_1|^{-\f12} \les t^{-1}.
	$$
	Here we used that $\lambda_0=|x-z_1|/2t$ in the
	last inequality.  
	
	In the second case one has $\lambda_0\les t^{-\f12}$,
	then $\lambda \les t^{-\f12}$, so that
	$$
		A \les \int_0^{t^{-\f12}} \frac{\lambda^{\f12}}{|x-z_1|^{\f12}} 
		\widetilde \chi(\lambda |x-z_1|)\, d\lambda
		\les t^{-\f34}|x-z_1|^{-\f12}.
	$$
	Here $\widetilde \chi =1-\chi$ is a cut-off away
	from zero which we employ to emphasize the support
	condition that 
	$\lambda |x-z_1|\gtrsim 1$.  For this integral to have
	a non-zero contribution, one must have
	$|x-z_1|^{-1}\les \lambda \les t^{-\f12}$, which then
	yields $A\les t^{-1}$ as desired.
	
	We now move to bounding $B$, the 
	first integral supported on
	$|\lambda-\lambda_0|\gtrsim t^{-\f12}$.  By Lemma~\ref{stat phase}, we need only show that
	$B\les 1$.  Again we 
	consider two cases.  First, if $\lambda_0\ll t^{-\f12}$, one sees that $|\lambda-\lambda_0|\approx 
	\lambda$.  So that
	$$
		B \les \int_{\R } \frac{\widetilde \chi(\lambda |x-z_1|)}{|x-z_1|^{\f12}\lambda^{\f32}}\, d\lambda
		\les |x-z_1|^{-\f12} \int_{|x-z_1|^{-1}}^\infty
		\lambda^{-\f32}\, d\lambda
		\les 1.
	$$
	In the second case one has $\lambda_0\gtrsim t^{-\f12}$.  In this case, we let $s=\lambda-\lambda_0$
	\begin{align*}
		B &\les \int_{|s|>t^{-\f12}} 
		\frac{(s+\lambda_0)^{\f12}}{|x-z_1|^{\f12}|s|^2}\, ds \les \frac{1}{|x-z_1|^{\f12}} \bigg(\int_{|s|>t^{-\f12}} s^{-\f32}+\lambda_0^{\f12} s^{-2}\, ds
		\bigg)\\
		&\les \frac{t^{\f14}}{|x-z_1|^{\f12}}+ \frac{t^{\f12}\lambda_0^{\f12}}{|x-z_1|^{\f12}}
		\les 1.
	\end{align*}
	The last inequality follows since $t^{-\f12}\les \lambda_0=|x-z_1|/2t$ implies that $t^{\f12}\les |x-z_1|$.

	We now turn to the final term $C$, we need only show
	$C\les 1$.  The first case is again when
	$\lambda_0\ll t^{-\f 12}$, in which case
	$|\lambda-\lambda_0|\approx \lambda$, and
	$$
		C\les \int_{\R} \frac{\widetilde \chi(\lambda|x-z_1|)}{\lambda^{\f32}|x-z_1|^{\f12}}
		\, d\lambda \les 1.
	$$
	In the second case $\lambda_0\gtrsim t^{-\f12}$, 
	which yields that $|x-z_1|\gtrsim t^{\f12}$.
	In this case,
	\begin{align*}
		C &\les \int_{|\lambda-\lambda_0|>t^{-\f12}}
		\frac{ \widetilde \chi(\lambda|x-z_1|)}{|x-z_1|^{\f12}
		\lambda^{\f12}|\lambda-\lambda_0|}\, d\lambda\\
		&\les |x-z_1|^{-\f12}\bigg(
		\int_{|\lambda-\lambda_0|>t^{-\f12}} \frac{ d\lambda}{|\lambda-\lambda_0|^{\f32}}+
		\int_{\R} \frac{ \widetilde \chi(\lambda|x-z_1|)}
		{\lambda^{\f32}}\, d\lambda
		\bigg) \les t^{\f14}|x-z_1|^{-\f12}+1 \les 1.
	\end{align*}

	We note that if the `+' phase is encountered instead
	of the `-', in place of \eqref{a osc int}, after
	again integrating by parts $\f n2-1$ times, one
	needs to bound
	\begin{align}\label{a osc int2}
		\frac{1}{|t|^{\f n2-1}} \int_0^\infty e^{it\lambda^2+i\lambda|x-z_1|} \chi(\lambda)
		a(\lambda) \, d\lambda
	\end{align}
	In which case, one can simply use that
	$\frac{d}{d\lambda}(e^{it\lambda^2+i\lambda|x-z_1|})
	=(2it\lambda+i|x-z_1|)e^{it\lambda^2+i\lambda|x-z_1|}$
	and integrate by parts.  The bound on $a(\lambda)$
	shows that the boundary terms are zero, so that
	$$
		|\eqref{a osc int2}| \les 
		|t|^{1-\f n2}\int_0^{\infty}
		\bigg|\frac{ \lambda^{-\f12}|x-z_1|^{-\f12}
		\widetilde \chi(\lambda|x-z_1|)}{2t\lambda+|x-z_1|}
		\bigg| \, d\lambda
		\les |t|^{-\f n2} \int_{\R } \frac{\widetilde \chi(\lambda |x-z_1|)}{\lambda^{\f32}|x-z_1|^{\f12}}
		\, d\lambda \les |t|^{-\f n2}.
	$$

The assumed decay rate on the potential is chosen
so that all spatial integrals are absolutely convergent.  
The analysis here is essentially the same as in the
odd dimensional case.
We note that
\begin{align}\label{eq:R0bs}
	|\partial_\lambda^j R_0^\pm(\lambda^2)(x,y)| \les 
	|x-y|^{j+2-n}+\lambda^{\frac{n-3}{2}}
	|x-y|^{j+\frac{1-n}{2}},
\end{align}
as developed in the proof of
Lemma~\ref{lem:R0exp}.  The second term decays more slowly
for large $x, y$, so it dictates the decay requirements for the potential.
In the iterated resolvent, differentiated $\f n2$ times,
we need to control integrals of the form
$$
	\int_{\R^{kn}}\frac{1}{|x-z_1|^{\frac{n-1}{2}-\alpha_0}}
	\prod_{j=1}^k \frac{V(z_j)}{|z_j-z_{j+1}|^{\frac{n-1}{2}-\alpha_j}}
	d\vec{z},
$$
where $\alpha_j\in \mathbb N_0$ and $\sum \alpha_j=\frac{n}{2}$, $z_{k+1}=y$ and
$d\vec z=dz_1\, dz_2\, \cdots \, dz_k$.
(There is a caveat  that
if $\alpha_0=\frac{n}{2}$ then the last derivative is applied
as in the stationary phase argument~\eqref{a osc int}
and does not yield a factor of $|x-z_1|^{\f 12}$ in the numerator.
Similarly if $\alpha_k=\frac{n}{2}$, the 
value of $\frac{n-1}{2} - \alpha_k$ should be treated as zero rather than $-\f 12$.)
Using arithmetic-geometric mean inequalities, 
any integral
we need to control is dominated by the sum
\begin{align*}
	\int_{\R^{kn}}\frac{1}{|x-z_1|^{\frac{n-1}{2}}}
	\prod_{j=1}^k \frac{V(z_j)}{|z_j-z_{j+1}|^{\frac{n-1}{2}}}
	\bigg(& |x-z_1|^{\frac{n-1}{2}}\\
	&+
	\sum_{\ell=2}^{k-1} |z_\ell-z_{\ell+1}|^{\frac{n}{2}} + |z_k-y|^{\frac{n-1}{2}}
	\bigg)
	d\vec{z}.
\end{align*}
Choose a representative element from the summation over $\ell$.  This negates a factor
of $|z_\ell - z_{\ell+1}|^{(1-n)/2}$ in the product and replaces it with
$|z_\ell - z_{\ell+1}|^{\f 12} \les \la z_\ell\ra^\f 12 \la z_{\ell+1}\ra^{\f 12}$.
With $|V(z_j)|\les \la z_j \ra^{-\beta}$, we have to control an integral
of the form
\begin{equation*}
\int_{\R^{kn}}\frac{1}{|x-z_1|^{\frac{n-1}{2}}}\biggl(\prod_{j=1}^{\ell-1}
\frac{\la z_j\ra^{-\beta}}{|z_j - z_{j+1}|^{\frac{n-1}{2}}}\la z_\ell\ra^{\f 12}\biggr)
\biggl(\la z_{\ell+1}\ra^{\f 12}\prod_{j=\ell}^k \frac{\la z_j\ra^{-\beta}}{|z_j - z_{j+1}|^{\frac{n-1}{2}}}\biggr)\,d\vec{z}
\end{equation*}
with $y = z_{k+1}$.
Assuming that $\beta > \frac{n+2}{2}$, this is bounded uniformly in $x,y$ by iterating the single integral
estimate
\begin{equation} \label{eq:simpiterated}
	\sup_{z_{j-1} \in \R^n} \int_{\R^n} \frac{\la z_j \ra^{\f12-\beta}}{|z_{j-1}-z_j|^{\frac{n-1}{2}}}
	\, dz_j \les 1,
\end{equation}
starting with $j = \ell$ we can iterate the above
bound and work outward the integating in $z_{\ell+1}$ to
$z_k$ and $z_{\ell-1}$ to $z_1$.

To make certain that the local singularities of the resolvent are integrable
uniformly in $x$ and $y$,
cancellation in the first factor $(R_0^+(\lambda^2) - R_0^-(\lambda^2))(x, z_1)$
is crucial.  By~\eqref{Jn canc}, this is a bounded function of the spatial variables.
Differentiation of resolvents with respect to $\lambda$ generally improves their 
local regularity, so for this purpose the worst case is when all derivatives act on the
cut-off function $\chi(\lambda)$ instead.
Then we are left to control an integral of the form
\begin{equation*}
 \int_{\R^{kn}} \biggl(\prod_{j=1}^{k-1} \frac{\la z_j \ra^{-\beta}}{|z_j - z_{j+1}|^{n-2}}\biggr)
\frac{\la z_k\ra^{-\beta}}{|z_k - y|^{n-2}}\,d\vec{z},
\end{equation*}
which is bounded so long as $\beta > 2$, using an estimate analogous to~\eqref{eq:simpiterated}.  We note that the
lack of the $|x-z_1|^{2-n}$ singular terms is vital to
this iterated integral being bounded for any $k=1,2,\dots$.
If the `+/-' difference acts on an inner resolvent,
say on $R_0^+(\lambda^2)(z_\ell,z_{\ell+1})-R_0^-(\lambda^2)(z_\ell,z_{\ell+1})$ we are lead to bound
\begin{equation*}
 \int_{\R^{kn}} \frac{1}{|x - z_1|^{n-2}}
 \biggl(\prod_{j=1}^{k-1} \frac{\la z_j \ra^{-\beta}}{|z_j - z_{j+1}|^{n-2}}\biggr)
\frac{\la z_k\ra^{-\beta}|z_\ell-z_{\ell+1}|^{n-2}}{|z_k - y|^{n-2}}\,d\vec{z},
\end{equation*}
Here,
one simply integrates $d\vec{z}$ first in the $z_\ell$ variable and proceed outward
through the rest of the product.


\end{proof}

We still need to consider the case in which all
derivatives act on the leading or lagging free
resolvent and the `+/-' difference affects a different
free resolvent, that is we wish to control the 
contribution of
\begin{align}\label{eq:innerdiff}
	\lambda \bigg[\bigg(\frac{1}{\lambda}\frac{d}
	{d\lambda}	\bigg)^{\f n2-1} R_0^-(\lambda^2)
	 \bigg] V (R_0^-(\lambda^2)V)^j  (R_0^+(\lambda^2)-
	 R_0^-(\lambda^2)) (VR_0^+(\lambda^2))^\ell,
	 \qquad j, \ell\geq 0
\end{align}
Here if we simply integrate by part the final time,
we have polynomial weights in the spatial variables when
the final derivative also acts on the leading free
resolvent.  As noted in the discussion preceeding
Lemma~\ref{lem:BS1}, this is somehow simpler than the
previous case.  In particular, the argument follows
using the techniques of the previous lemma, and
the resulting calculation is streamlined using the following Lemma.  We first  define $\mathcal G_n (\lambda,|x-y|)$ to be the
kernel of the $n$-dimensional free resolvent operator
$R_0^+(\lambda^2)$, and hence $\mathcal G_n (-\lambda,|x-y|)$ is the kernel of $R_0^-(\lambda^2)$,
 then
    \begin{lemma}\label{recurrence}
            For $n\geq 2$, the following recurrence relation holds.
        \begin{align*}
            \left(\frac{1}{\lambda}\frac{d}{d\lambda}\right) \mathcal G_n       
            (\lambda, r)
            =\frac{1}{2\pi}\mathcal G_{n-2}(\lambda,r).
        \end{align*}
    
    \end{lemma}
    
    \begin{proof}
        The proof follows from the recurrence relations of the Hankel functions,
        found in \cite{AS} and the representation of the kernel given in
        \eqref{Hankel}.
        
    \end{proof}

This tells us that the
action of $\frac{1}{\lambda}\frac{d}{d\lambda}$ takes
an $n$-dimensional free resolvent to an $n-2$ dimensional
free resolvent.
With this, we are now ready to prove
\begin{lemma}\label{lem:BS2}

	If $|V(x)|\les \la x\ra^{-\frac{n+2}{2}-}$ and 
	$j,\ell \geq 0$,
	we have the bound
	\begin{align*}
		\sup_{x,y\in \R^n} 
		\bigg|\int_0^\infty e^{it\lambda^2}
		\lambda \chi(\lambda) 
		[R_0^-(\lambda^2)V(R_0^-(\lambda^2)V)^j [R_0^+- R_0^-](\lambda^2) (VR_0^+(\lambda^2))^\ell](x,y)\, d\lambda\bigg|
		\les |t|^{-\f n2}.
	\end{align*}

\end{lemma}

\begin{proof}
As in the proof of Lemma~\ref{lem:BS1}, 
we need only consider the case when all the
derivatives act on the leading resolvent.  The other
cases are less delicate and can be treated identically.

At this point, by using Lemma~\ref{recurrence} a total of $\frac{n}{2}-1$ times the leading free resolvent is
a constant multiple of the two-dimensional free
resolvent.  Thus, we can
we can reduce the contribution of \eqref{eq:innerdiff} to
$$
	t^{1-\f n2}\int_0^\infty e^{it\lambda^2} \chi(\lambda)
	\lambda (iJ_0(\lambda|x-\cdot|)+Y_0(\lambda|x-\cdot|))V
	\widetilde O_1(\lambda^{n-2})\, d\lambda.
$$
The Bessel functions of order zero appear as the kernel
of a two-dimensional resolvent.  The $\widetilde{O}_1(\lambda^{n-2})$
expression is much smaller than necessary ($\widetilde{O}_1(\lambda^{0+})$
would be adequate), so it can absorb singularities of the Bessel functions
with respect to $\lambda$.

 Expansions for these
Bessel functions, see \cite{AS}, \cite{Sc2} or \cite{EG}, show
that for $\lambda|x-z_1|\ll 1$,
\begin{align*}
	|iJ_0(\lambda|x-z_1|)+Y_0(\lambda|x-z_1|)|
	&=1+\log(\lambda |x-z_1|)+\widetilde O_1((\lambda|x-z_1|)^{2-}),\\
	|\partial_\lambda [iJ_0(\lambda|x-z_1|)+Y_0(\lambda|x-z_1|)]|
	&=\lambda^{-1}+\widetilde O_1((\lambda|x-z_1|)^{1-})
\end{align*}
Recall that
$$
	|(1+\log(\lambda|x-z_1|))\chi(\lambda|x-z_1|)
	\chi(\lambda)|\les 1+ |\log \lambda|+\log^- |x-z_1|
$$
The $\log \lambda = \widetilde{O}_1(\lambda^{0-})$ singularity
is easily negated by $\widetilde{O}_1(\lambda^{n-2})$ as mentioned above.
The $\log^- |x-z_1|$
singularity is integrable, and is managed by the estimate
\begin{equation*}
\sup_{x \in \R^n} \int_{\R^n} \log^-|x-z_1|\la z_1\ra^{-\beta}\, dz_1 \les 1
\end{equation*}
for any $\beta > n$.


For $\lambda|x-y|\gtrsim 1$, one has the description
\begin{equation*}
iJ_0(\lambda|x-z_1|) + Y_0(\lambda|x-z_1|) =
e^{i\lambda|x-z_1|}\omega_+(\lambda|x-z_1|)+ e^{-i\lambda|x-z_1|}\omega_-(\lambda|x-z_1|)
\end{equation*}
similar in form to \eqref{Jn asymp} but with  different
functions $\omega_\pm(z)$.  Differentiating directly with respect to $\lambda$
is not advised, as
the resulting $|x-z_1|\omega_\pm(\lambda|x-z_1|)$ term grows like $\lambda^{-\f 12}\la x\ra^{\f 12}$
for large $x$.

However this issue was encountered once before while evaluating~\eqref{a osc int}.
The same argument from Lemma~\ref{lem:BS1} applies here as well and
yields the desired unweighted bound, again with more
than enough $\lambda$ smallness to ensure the argument
runs through.


\end{proof}

This provides all we need for the proof of the main
proposition in this section.

\begin{proof}[Proof of Proposition~\ref{bsprop}]

	The proposition follows from Lemma~\ref{lem:BS1},
	the discussion following this Lemma and finally
	from Lemma~\ref{lem:BS2}.

\end{proof}

\section{Dispersive estimates: the leading terms}\label{sec:sing}

In this section we prove dispersive bounds for the
most singular $\lambda$ terms of the expansion for
$R_V^+(\lambda^2)-R_V^-(\lambda^2)$.  These terms are
sensitive to the existence of zero energy eigenvalues
and are the slowest decaying in time. 
This behavior  arises in the
last term involving the operator $M^{\pm}(\lambda)^{-1}$
in \eqref{eq:bstail}.

From the `+/-' cancellation, we need to control the
contribution of
\begin{equation}\label{eq:Minv canc}
\begin{aligned}
	(R_0^+(\lambda^2)V)^mR_0^+(\lambda^2) v
	&M^+(\lambda)^{-1}v R_0^+(\lambda^2) (V
	R_0^+(\lambda^2))^m\\
	&-(R_0^-(\lambda^2)V)^mR_0^-(\lambda^2) v
	M^-(\lambda)^{-1}v R_0^-(\lambda^2) (V
	R_0^-(\lambda^2))^m
\end{aligned}
\end{equation}
to the Stone formula, \eqref{Stone}.  Thanks to the
algebraic fact \eqref{alg fact}, we need to consider
three cases.  The difference of `+' and `-' terms may
act on the operators $M^\pm(\lambda)^{-1}$ or on the
free resolvents.  As in the treatment of the finite
Born series terms in Section~\ref{sec:BS},
if the difference acts on free resolvents
we need to distinguish if they are `inner' resolvents which
require less care than the case of `leading' or `lagging' resolvents.

\subsection{No cancellation}

We first consider the case in which there are no 
cancellation properties to take advantage of, that 
is when $P_eV1\neq 0$.

\begin{lemma}\label{lem:nocanc}

	If $P_eV1\neq 0$ and $|V(x)|\les \la x\ra^{-n-}$,
	then
	\begin{align*}
		\eqref{eq:Minv canc}&= \lambda^{n-6} P_eV1VP_e
		+ \widetilde{O}_{\f n2 - 1}(\lambda^{n-6+})
	\end{align*}
	which contributes $c|t|^{2-\f n2}P_eV1VP_e+ 
	O(|t|^{2-\f n2 +})$ to \eqref{Stone}.

	If $P_eV1\neq 0$ and $|V(x)|\les \la x\ra^{-n-4-}$,
	then
	\begin{align*}
		\eqref{eq:Minv canc}&= \lambda^{n-6} P_eV1VP_e
		+\frac{R_0^+(\lambda^2)-R_0^-(\lambda^2)}
		{\lambda^2} VP_e
		+P_eV\frac{R_0^+(\lambda^2)-R_0^-(\lambda^2)}
		{\lambda^2}
		+ \mathcal E(\lambda)
	\end{align*}
	which contributes $c|t|^{2-\f n2}P_eV1VP_e+ 
	O(|t|^{1-\f n2})$ to \eqref{Stone}.

\end{lemma}

Here we cannot write the final error term $\mathcal E(\lambda)$ accurately as
$\widetilde O_k(\lambda^\alpha)$, as there are too
many fine properties of this error term that this
notation fails to capture if one hopes to attain the
faster $|t|^{1-\f n2}$ decay rate.  One can explicitly
reconstruct $\mathcal E(\lambda)$ from our proof, though
we do not think it worthwhile to do so.

We note that the terms 
$$
	\frac{R_0^+(\lambda^2)-R_0^-(\lambda^2)}
	{\lambda^2} VP_e
	+P_eV\frac{R_0^+(\lambda^2)-R_0^-(\lambda^2)}
	{\lambda^2}
$$
appear in the expansion in all cases, see the statements of
Lemmas~\ref{lem:nocanc}, \ref{lem:1canc} and
\ref{lem:3canc}.  The different
cancellation assumptions on $P_eV1$ and $P_eVx$ allow us
some flexibility on how to control their contribution 
to \eqref{Stone}.  To avoid presenting three proofs of
how to bound these terms, which would have a certain amount of
overlap, we control these terms separately in Lemma~\ref{lem:2canc}
and Corollary~\ref{cor:2canc} below.

\begin{proof}

The first statement is a straightforward application of Lemma~\ref{lem:R0exp}
and Corollary~\ref{cor:Minv no cancel} in the context of applying \eqref{alg fact}
to~\eqref{eq:Minv canc}.
Lemmas~\ref{lem:IBP} and~\ref{lem:fauxIBP} then control the respective
integrals in~\eqref{Stone} due to the leading term and the remainder.

More precisely, the leading term appears if the `+/-' difference in~\eqref{alg fact}
falls on $M^\pm(\lambda)^{-1}$.  In that case Corollary~\ref{cor:Minv no cancel}
indicates that 
\begin{align*}
M^+(\lambda)^{-1} - M^-(\lambda)^{-1} &= \frac{g_1^+(\lambda) - g_1^-(\lambda)}{\lambda^4}
M_{n-6}^L + \widetilde{O}_{\f n2 - 1}(\lambda^{n-6+}) \\
&= 2\Im (z_1)\lambda^{n-6}M_{n-6}^L + \widetilde{O}_{\f n2 -1}(\lambda^{n-6+}),
\end{align*}
where we used~\eqref{eq:g diff} in the last line. 
Meanwhile $R_0^\pm(\lambda^2) = G_0^0 + \widetilde{O}_{\f n2 - 1}(\lambda^{0+})$.
Together with the fact that $V$ is integrable, this establishes the remainder as
$\widetilde{O}_{\f n2 - 1}(\lambda^{n-6+})$.  The operator in the leading term is seen,
using identities \eqref{S1 ident}, \eqref{D1 useful} and \eqref{Mn-6L}, to be
\begin{equation*}
(G_0^0V)^mG_0^0 vM_{n-6}^LvG_0^0(VG_0^0)^m
	=(G_0^0V)^mG_0^0 vD_1vG_{n-2}^cvD_1 vG_0^0(VG_0^0)^m=P_e V1VP_e.
\end{equation*}

If the +/- difference acts on any one of the resolvents in~\eqref{eq:Minv canc}, we see
that $R_0^+(\lambda^2) - R_0^-(\lambda^2) = \widetilde{O}_{\f n2 - 1}(\lambda^{n-2})$,
$R_0^\pm(\lambda^2)(z_j,z_{j+1})=(1+\log^-|z_j-z_{j+1}|)\widetilde O_{\f n2-1}(1)$
and $M^\pm(\lambda)^{-1} = \widetilde{O}_{\f n2 -1}(\lambda^{-2})$.  Recall that the notation
$\widetilde O_{\f n2-1}(1)$ indicates that differentiation
in $\lambda$ is comparable to division by $\lambda$.
That more than
suffices to place all of these terms in the remainder.

Now assume that $|V(x)| \les \la x \ra^{-n-4-}$. Carrying out the power series expansion
further in Corollary~\ref{cor:Minv no cancel}, one obtains
\begin{multline*}
	M^+(\lambda)^{-1}-M^-(\lambda)^{-1}=
	\frac{g_1^+(\lambda)-g_1^-(\lambda)}
	{\lambda^4}M_{n-6}^L
	+\frac{g_1^+(\lambda)-g_1^-(\lambda)}{\lambda^2}
	M_{n-4}^{L1}\\
	+\frac{g_2^+(\lambda)-g_2^-(\lambda)}
	{\lambda^4}
	M_{n-4}^{L2}
	+\widetilde{O}_{\f n2-1}
	(\lambda^{n-4+})\\
	=2\Im (z_1)\lambda^{n-6}M_{n-6}^L
	+2\Im (z_1) \lambda^{n-4} M_{n-4}^{L1}
	+ 2\Im (z_2)
	\lambda^{n-4}M_{n-4}^{L2}+\widetilde{O}_{\f n2-1}
	(\lambda^{n-4+}).
\end{multline*}
Similarly, we have $R_0^\pm(\lambda^2)(x,y) = G_0^0 + \lambda^2G_1^0
+ (1 + \log^-|x-y|)\widetilde{O}_{\f n2 - 1}(\lambda^4)$.  Thus the term
featuring $M^+(\lambda)^{-1} - M^-(\lambda)^{-1}$ has the form
\begin{equation} \label{Minv diff}
\begin{aligned}
	(G_0^0&V)^mG_0^0 v[M^+(\lambda)^{-1}-M^-(\lambda)^{-1}]
	vG_0^0(VG_0^0)^m\\
	&+[\lambda^2 \Gamma_1
		+(1+\log^- |x-\cdot|)\widetilde O_{\f n2-1}(\lambda^4)][M^+(\lambda)^{-1}-M^-(\lambda)^{-1}]
	vG_0^0(VG_0^0)^m\\
	&+(G_0^0V)^mG_0^0v[M^+(\lambda)^{-1} - M^-(\lambda^{-1}][\lambda^2 \Gamma_1
		+(1+\log^- |x-\cdot|)\widetilde O_{\f n2-1}(\lambda^4)] \\
	&+[\lambda^2 \Gamma_1
		+(1+\log^- |x-\cdot|)\widetilde O_{\f n2-1}(\lambda^4)][M^+(\lambda)^{-1}-M^-(\lambda)^{-1}]\\
	&\qquad \times [\lambda^2 \Gamma_1
	+(1+\log^-|\cdot-y|)\widetilde O_{\f n2-1}(\lambda^4)]\\
	&\hskip -.25in =\lambda^{n-6} P_eV1VP_e+ \lambda^{n-4}K_1+
	\widetilde O_{\f n2 -1}(\lambda^{n-4+})K_2
\end{aligned}
\end{equation}
with $K_1, K_2$ operators that map $L^1\to L^\infty$.

If the +/- difference falls on a free resolvent in the interior of the product, we have
\begin{equation*}
(R_0^+(\lambda^2) - R_0^-(\lambda^2))(z_j, z_{j+1}) =
c\lambda^{n-2}G_{n-2}^c + |z_j-z_{j+1}|^{0+}\widetilde{O}_{\f n2 - 1}(\lambda^{n-2+})
\end{equation*}
and $M^\pm(\lambda)^{-1} = -\lambda^{-2} D_1 +  \widetilde{O}_{\f n2 - 1}(1)$.  The resulting term of~\eqref{eq:Minv canc}
takes the form $\lambda^{n-4} K_3 + \widetilde{O}_{\f n2 - 1}(\lambda^{n-4+})$, with $K_3$ another operator from $L^1$ to $L^\infty$.

We note that the extra power
of $|z_j-z_{j+1}|^{0+}$ that appears in the remainder term is acted on by $R_0^-(\lambda^2)V$ on the left
and $VR_0^+(\lambda^2)$ on the right, so that the
decay of the potentials ensures that the product remains bounded between unweighted spaces.

The terms in which the `+/-' difference
acts on the first (or last) free resolvent are
trickier because one cannot differentiate too many times, or go too far into the power series expansion
of $R_0^+(\lambda^2)- R_0^-(\lambda^2)$
without introducing weights.  Suppose the difference acts on the leading resolvent;
the other case is identical up to symmetry.
Once again we can use the expansions for $M^\pm(\lambda)^{-1}$ and
$R_0^\pm(\lambda^2)$ along with Lemma~\ref{lem:Jn est}
to express this term as
\begin{align*}
	[R_0^+(\lambda^2)-R_0^-(\lambda^2)](VG_0^0)^m v
	\bigg(-\frac{D_1}{\lambda^2}
	\bigg)v (G_0^0V)^mG_0^0 &+ \widetilde{O}_{\f n2 - 1}(\lambda^{n-2})  \\
	= \frac{R_0^+(\lambda^2) - R_0^-(\lambda^2)}{\lambda^2}VP_e &+ \widetilde{O}_{\f n2 - 1}(\lambda^{n-2})
\end{align*}
One can quickly show using Lemma~\ref{lem:fauxIBP} that the remainder contributes at most
$|t|^{1-\f n2}$ to the Stone formula.  In fact this contribution is of the order $|t|^{-\f n2}$,
seen by adopting the methods of Lemma~\ref{lem:BS1}.  The contribution of $\lambda^{-2}(R_0^+(\lambda^2)-R_0^-(\lambda^2))VP_e$
to~\eqref{Stone} is rather intricate, and is discussed fully as Lemma~\ref{lem:2canc}.  For the purpose of this
Lemma, we note that  $\lambda^{-2}(R_0^+(\lambda^2)-R_0^-(\lambda^2))VP_e$
is bounded by $|t|^{1-\f n2}$ as an operator from 
$L^1\to L^\infty$ by Lemma~\ref{lem:2canc}, which finishes the proof.

\end{proof}

The remaining terms in the Born series are smaller than
these for large $|t|$ by Proposition~\ref{bsprop}.
In fact using the identities for $S_1$ and
Lemma~\ref{lem:eproj}, at this point we can write
\begin{align}
	e^{itH}P_{ac}(H)=c|t|^{2-\f n2} P_eV
	1VP_e+O(|t|^{1-\f n2}),
\end{align}
where the operator $P_eV1VP_e$ is rank one, and the
error term is understood as mapping $L^1$ to $L^\infty$.  The weaker
claim, with error term of size $o(|t|^{2-\f n2})$ follows
by using the first statement of Lemma~\ref{lem:1canc}.

\subsection{The case of $P_eV1=0$}

Here we consider when the operator $P_eV1=0$.  This
cancellation makes the initial term in Lemma~\ref{lem:nocanc} vanish,
clearing the way for time decay
at the faster rate of $|t|^{1-\f n2}$.  Here we provide more detail
on the behavior of the next term in the evolution.

\begin{lemma}\label{lem:1canc}

	If $P_eV1= 0$ and $|V(x)|\les \la x\ra^{-n-8-}$,
	then
	\begin{align*}
		\eqref{eq:Minv canc}&=\lambda^{n-4}\Gamma_1+
		\frac{R_0^+(\lambda^2)-R_0^-(\lambda^2)}
		{\lambda^2} VP_e
		+P_eV\frac{R_0^+(\lambda^2)-R_0^-(\lambda^2)}
		{\lambda^2}+\lambda^{n-2}\Gamma_2+
		\mathcal E(\lambda),
	\end{align*}
	where $\Gamma_1, \Gamma_2:L^1\to L^\infty$.  The
	error term belongs to the class 
	$\la x\ra^{\f 12}\la y\ra^{\f 12}\widetilde{O}_{\f n2}(\lambda^{n-2+})$, however
	its contribution to~\eqref{Stone} is $O(|t|^{-\f n2})$ without spatial weights.
	Assuming the result of Lemma~\ref{lem:2canc}, the total contribution to~\eqref{Stone}
	of all terms is $|t|^{1-\f n2}+ \la x\ra \la y\ra O(|t|^{-\f n2})$.

\end{lemma}

We note that the error term $\mathcal E(\lambda)$
here is distinct from the
error term in Lemma~\ref{lem:nocanc}.  

\begin{proof}

The structure of the argument is the same as in the preceding lemma.  The extra decay
permits us to evaluate more terms of each power series, or better control the remainder.
The fact that $P_eV1 = 0$ causes some of the leading order expressions to vanish.

When the `+/-' cancellation in~\eqref{alg fact} acts on
$M^{\pm}(\lambda)^{-1}$, the first nonzero term has size
$\lambda^{n-4}$.  In detail, we note that
by Corollary~\ref{cor:Minv}, specifically \eqref{Minv PV102} we have
\begin{align*}
	M^+(\lambda)-M^-(\lambda)&=
	\frac{g_2^+(\lambda)-g_2^-(\lambda)}{\lambda^4}
	M_{n-4}^{L2}
	+\frac{g_2^+(\lambda)-g_2^-(\lambda)}{\lambda^2}
	M_{n-2}^{L2}
	+\widetilde{O}_{\f n2}(\lambda^{n-})\\
	&= c_1 \lambda^{n-4}M_{n-4}^{L2}+c_2 \lambda^{n-2}
	M_{n-2}^{L2}+\widetilde O_{\f n2}(\lambda^{n-2+})
\end{align*}
Writing the resolvents as
$R_0^{\pm}(\lambda^2)(x,y)=G_0^0+\lambda^2 G_1^0 +
\la x\ra^{\f 12}\la y \ra^{\f 12}\widetilde O_{\f n2}(\lambda^4)$, as suggested by Corollary~\ref{cor:R0exp},
we can see that
\begin{align*}
	[R_0^-(\lambda^2)V]^m& R_0^-(\lambda^2)v
	[M^+(\lambda)-M^-(\lambda)]v R_0^+(\lambda^2)
	[VR_0^+(\lambda^2)]^m(x,y)\\
	&=c_1\lambda^{n-4}(G_0^0V)^mG_0^0v M_{n-4}^{L2}v
	G_0^0 (VG_0^0)^m +\lambda^{n-2} K_2
	+\la x\ra^{\f 12}\la y \ra^{\f 12}\widetilde 
	O_{\f n2}(\lambda^{n-2+})
\end{align*}
Here $K_2$ is a finite rank operator made out $G_0^0 $'s
and $vM_{n-2}^{L2}v$ along with all the combinations consisting of
$G_0^0$'s, $vM_{n-4}^{L2}v$ and exactly one instance of $G_1^0$.
Lemma~\ref{lem:IBP}  shows that
the first term contributes $|t|^{1-\f n2}$ to 
\eqref{Stone} and the second term contributes $|t|^{-\f n2}$.
Lemma~\ref{lem:fauxIBP2} shows that the last term
generates a map from $L^{1, \f 12}$ to $L^{\infty, -\f 12}$ with norm
$|t|^{-\f n2}$.  The half-power weights only arise if one allows $\f n2$ derivatives
to fall on the first or the last free resolvent in the product.  The argument in Lemma~\ref{lem:BS2} of
using the stationary phase bound of Lemma~\ref{stat phase}
in place of the last integration by parts
shows how that situation can be prevented, so that all the expressions with time decay $|t|^{-\f n2}$
are bounded operators from $L^1$ to $L^\infty$.

Now suppose the `+/-' difference acts on a
free resolvent in the interior of the product.
We may write
\begin{align*}
[R_0^+(\lambda^2)-R_0^-(\lambda^2)](z_j,z_{j+1})&=
	\lambda^{n-2}G_{n-2}^c +\lambda^{n}G_{n}^c +
      |z_j - z_{j+1}|^{2+}\widetilde{O}_{\f n2}(\lambda^{n+}) ,\\
 R_0^\pm(\lambda^2)(z_j,z_{j+1}) &= G_0^0 + \lambda^2G_1^0 
  + \la z_j\ra^{\f 12}\la z_{j+1}\ra^{\f 12}\widetilde{O}_{\f n2}(\lambda^{2+}),\\
 M^\pm(\lambda)^{-1} &= -\lambda^{-2}D_1 + M_0 + \widetilde{O}_{\f n2}(\lambda^{0+}).
\end{align*}
Note that $P_eV1 = 0$ causes the leading term ($\lambda^{n-4}K_3$ in the previous lemma)
to vanish because  $(VG_0^0)^{m-j}D_1 = VP_ew$
and $G_{n-2}^c(z_j, z_{j+1})=c_{n-2}1$ is a constant function.  Thus $G_{n-2}^c(VG_0^0)^{m-j}D_1 = 0$.

Expressions with $\lambda^{n-2}$ occur by replacing the leading term in exactly one of the above
power series by its successor.  That is when $\lambda^n G_n^c$ occurs in place of $\lambda^{n-2}G_{n-2}^c$,
$\lambda^2 G_1^0$ in place of $G_0^0$ or
$M_0$ in place of $-\lambda^2 D_1$.  The operator $G_n^c$ has spatial growth
of $|z_j - z_{j+1}|^2$ but it is controlled by the
decay of the potentials as it is multiplied on both sides by $V(z_j)$ and $V(z_{j+1})$.

Remainders in the class $\widetilde{O}_{\f n2}(\lambda^{n-2+})$
are mostly bounded from $L^1$ to $L^\infty$ as well, except that once again weights of
$\la x\ra^{\f 12}$ or $\la y\ra^{\f 12}$ arise if all $\f n2$ derivatives
fall on the first or the last free resolvent.  Following the calculations in Lemma~\ref{lem:BS2},
one can see that the contribution of these remainder terms to~\eqref{Stone} has time decay
 $|t|^{-\f n2}$ as a map between unweighted $L^1$ and $L^\infty$.

Now suppose the difference of free resolvents occurs at the leading resolvent of the product~\eqref{eq:Minv canc}.
The expression where one approximates all other free resolvents by $G_0^0$, and $M^+(\lambda)^{-1}$
by $-\lambda^{-2}D_1$, is considered separately in Lemma~\ref{lem:2canc}.  Under the assumption
$P_eV1 = 0$, its contribution to~\eqref{Stone} is an operator with kernel bounded by $\la x\ra|t|^{-\f n2}$.
The analogous expression when the +/- difference is applied to the very last resolvent in the product
yields a bound of $\la y\ra|t|^{-\f n2}$.  Put together, these operators form a map from $L^{1,1}$
to $L^{\infty, -1}$ with time decay $|t|^{-\f n2}$.

Finally there is an assortment of remainder terms found by applying~\eqref{alg fact} to
\begin{align*}
\big(R_0^+(\lambda^2) - R_0^-(\lambda^2)\big)\Big[
	(VR_0^+(\lambda^2)V)^m v
	M^+(\lambda)^{-1}&v R_0^+(\lambda^2) (V
	R_0^+(\lambda^2))^m\\
	&-(VG_0^0)^m v
	\Big(-\frac{D_1}{\lambda^2}\Big)v G_0^0
	(V G_0^0)^m\Big].
\end{align*}
Each one is headed by $(R_0^+(\lambda^2) - R_0^-(\lambda^2))$,
concludes with either $R_0^+(\lambda^2)$ or $G_0^0$, and
is of order $\lambda^{n-2}$.  Following the calculations in Lemma~\ref{lem:BS1}
one can show that they contribute $|t|^{-\f n2}$ to~\eqref{Stone}.

\end{proof}

%

Hence we have if $P_eV1=0$
\begin{align*}
	e^{itH}P_{ac}(H)=|t|^{1-\f n2}\Gamma+O(|t|^{-\f n2})
\end{align*}
where $\Gamma$ is a finite rank operator mapping
$L^1$ to $L^\infty$, which we
do not make explicit and the error term is understood
as an operator between weighted spaces.  Combining this
with the analysis for when $P_eV1\neq 0$, we have the
expansion
\begin{align*}
	e^{itH}P_{ac}(H)=c|t|^{2-\f n2}P_eV1VP_e+
	|t|^{1-\f n2}\Gamma_2+O(|t|^{-\f n2}),
\end{align*}
with $\Gamma_2:L^1\to L^\infty$ a finite rank operators,
which is valid whether or not $P_eV1=0$.

\subsection{The case of $P_eV1=0$ and
$P_eVx=0$}

Finally we consider the evolution when we have both
cancellation conditions on the zero-energy eigenfunctions.

\begin{lemma}\label{lem:3canc}

	If $P_eV1= 0$, $P_eVx=0$ and $|V(x)|\les \la x\ra^{-n-8-}$,
	then
	\begin{align*}
		\eqref{eq:Minv canc}&=
		\frac{R_0^+(\lambda^2)-R_0^-(\lambda^2)}
		{\lambda^2} VP_e
		+P_eV\frac{R_0^+(\lambda^2)-R_0^-(\lambda^2)}
		{\lambda^2}+\lambda^{n-2}\Gamma_3+
		\mathcal E(\lambda),
	\end{align*}
	where $\Gamma_3:L^1\to L^\infty$.  The
	error term contributes $O(|t|^{-\f n2})$ as an
	operator from $L^1\to L^\infty$.  Assuming the result of
	Lemma~\ref{lem:2canc}, the total contribution to~\eqref{Stone}
	of all terms is $O(|t|^{-\f n2})$.

\end{lemma}

Again the error term $\mathcal E(\lambda)$ is distinct
from the previous lemmas.

\begin{proof}

As in the proofs of Lemmas~\ref{lem:nocanc} and
\ref{lem:1canc} we have to consider when the `+/-' 
difference in \eqref{alg fact} acts on either a 
resolvent of $M^\pm (\lambda)^{-1}$.  In the latter case,
the same argument as above goes through, though we
note (from Corollary~\ref{cor:Minv}) that the operator
$M_{n-4}^{L2}=0$, so that
\begin{align*}
	M^+(\lambda)-M^-(\lambda)&=
	\frac{g_3^+(\lambda)-g_3^-(\lambda)}{\lambda^4}
	M_{n-2}^{L3}
	+\widetilde{O}_{\f n2}(\lambda^{n-2+})
	=c_2 \lambda^{n-2}
	M_{n-2}^{L3}+\widetilde O_{\f n2}(\lambda^{n-2+})
\end{align*}
This easily gives us the bound of $|t|^{-\f n2}$ when
combined with the previous sections as an operator
from $L^{1}$ to $L^{\infty}$.

When the `+/-' difference acts on free resolvents, 
we can control the contribution by $|t|^{-\f n2}$ as an
operator from $L^1\to L^\infty$ if the difference acts
on an `inner' resolvent as before.   
For the remaining two terms, when the `+/-' acts on a
leading or lagging free resolvent,
we use the following estimates of Lemma~\ref{lem:2canc}.

\end{proof}


\begin{lemma}\label{lem:2canc}

	The operator
	$$
		\frac{R_0^+(\lambda^2)-R_0^-(\lambda^2)}{\lambda^2}
		VP_e
	$$
	contributes $|t|^{1-\f n2}$ to \eqref{Stone} as an operator from
	$L^1$ to $L^\infty$. 
	If $P_eV1 = 0$, then it contributes
	$|t|^{-\f n2}$ as an operator from $L^1$ to $L^{\infty,-1}$.
	If in addition $P_eVx=0$, then the contribution still has size $|t|^{-\f n2}$,
	but acts as an operator from $L^1$ to $L^\infty$.

\end{lemma}

Here we need to be careful with the spatial variables
to see that the orthogonality conditions allow us to 
move the dependence on $x$ or $y$ into an inner
spatial variable, which can be controlled by the
decay of the potential.
To make this clear, we note that we wish to
bound the integral
\begin{align}\label{R0Pe}
\int_0^1 e^{it\lambda^2} \chi(\lambda) \lambda^{-1}
(R_0^+(\lambda^2) - R_0^-(\lambda^2))(x,z_1)V(z_1)P_e
(z_1,y) \,d\lambda
\end{align}
in terms of $t,x$ and $y$.

To prove this lemma, we first need to following
oscillatory integral estimate, whose proof is in
Section~\ref{sec:Spec}.

\begin{lemma} \label{lem:off_center_statphase}
Let $m$ be any positive integer.
Suppose $|\Omega^{(k)}(z)| \leq \la z \ra^{\frac{1-m}{2}-k}$ for each $k \geq 0$.
Then
\begin{equation}
\int_0^\infty e^{it\lambda^2} \lambda^{m-1}e^{\pm i\lambda r}\Omega(\lambda r)\chi(\lambda)\,d\lambda
\les |t|^{-\frac{m}{2}}
\end{equation}
with a constant that does not depend on the value of $r > 0$.
\end{lemma}

We note that $m$ in this lemma is an arbitrary integer, not that value chosen in \eqref{eq:bstail} that ensures the iterated resolvents are locally $L^2$.

\begin{proof}[Proof of Lemma~\ref{lem:2canc}]

%
%
According to~\eqref{Jn canc},
 the integral kernel of $R_0^+(\lambda^2) - R_0^-(\lambda^2)$
can be expressed (modulo constants) as
\begin{multline*}
K(\lambda, |x-z_1|) = \lambda^{n-2} \frac{J_{\frac{n}{2}-1}(\lambda|x-z_1|)}
{(\lambda|x-z_1|)^{\frac{n}{2}-1}}\\
= \lambda^{n-2} \big(e^{i\lambda|x-z_1|}\Omega_+(\lambda|x-z_1|) 
+ e^{-i\lambda|x-z_1|}\Omega_-(\lambda|x-z_1|)\big),
\end{multline*}
where the functions $\Omega_\pm$ and their derivatives satisfy
$|\Omega^{(k)}_\pm(z)| \les \la z \ra^{\frac{1-n}{2} - k}$.
Derivatives with respect to the spatial variable $r = |x-z_1|$ are obtained by differentiating \eqref{Jn low} and
\eqref{Jn asymp} according to whether $\lambda r$ is small or large.  
Since the expansion of $z^{1-\frac{n}{2}}J_{\frac{n}{2}-1}(z)$ in~\eqref{Jn low}
has only even powers of $z$, its first derivative is bounded by $|z|$ rather than a constant.  Thus we can write
\begin{equation} \label{eq:Omega1-2}
\begin{aligned}
\partial_rK(\lambda, r) &= \lambda^{n}r\big(e^{i\lambda r}\Omega_{1,+}(\lambda r)
+ e^{-i\lambda r} \Omega_{1,-}(\lambda r)\big) \\
\partial^2_rK(\lambda, r) &= \lambda^{n}\big(e^{i\lambda r}\Omega_{2,+}(\lambda r)
+ e^{-i\lambda r} \Omega_{2,-}(\lambda r)\big)
\end{aligned}
\end{equation}
where $|\Omega_{j,\pm}^{(k)}(z)| \les \la z \ra^{\frac{1-n}{2}-j-k}$ for $j = 1, 2$
and all $k \geq 0$.  

Roughly speaking, the bound on
$\partial_r K(\lambda,r)$ gains two powers of $\lambda$
at the cost of one power of $r=|x-z_1|$.  This gains us an extra
power of time decay in the contribution to the Stone
formula, \eqref{Stone}, at the cost of one power spatial
weight.   The bound on $\partial_r^2 K(\lambda, r)$ 
allows us to gain the desired time decay with no spatial weights.

As an immediate consequence we can apply Lemma~\ref{lem:off_center_statphase} with $m = n-2$
to obtain
\begin{equation*}
\int_0^\infty e^{it\lambda^2} \lambda^{n-3} e^{\pm i \lambda|x-z_1|} \Omega_\pm(\lambda|x-z_1|)
\chi(\lambda)\,d\lambda \les |t|^{1-\frac{n}{2}}
\end{equation*}
and therefore $\int_0^\infty e^{it\lambda^2}\lambda^{-1}\chi(\lambda)
(R_0^+(\lambda^2) - R_0^-(\lambda^2))VP_e\,d\lambda$ maps $L^1$ to $L^\infty$ with norm
decay of $|t|^{1-\frac{n}{2}}$.

When $P_eV1=0$,
we can extract a leading-order term by replacing $K(\lambda,|x-z_1|)$ by
$K(\lambda,|x-z_1|)-K(\lambda,|x|)$ each place that it occurs.  From an operator perspective this
amounts to approximating $R_0^+(\lambda^2) - R_0^-(\lambda^2)$ by
$K(\lambda|x|)1$.  This term vanishes from the Schr\"odinger evolution precisely
when $P_eV1 = 0$.

The remainder can be written using the expression
\begin{equation*}
K(\lambda, |x-z_1|) - K(\lambda, |x|) = \int_0^1 \partial_r K(\lambda, |x-sz_1|)  \frac{(-z_1) \cdot (x-sz_1)}{|x-sz_1|}\, ds.
\end{equation*}
Based on the decomposition in~\eqref{eq:Omega1-2} and Lemma~\ref{lem:off_center_statphase} with $m = n$, we have the bound
\begin{equation*}
\Big|\int_0^\infty e^{it\lambda^2} \lambda^{-1} \chi(\lambda)
\partial_r K(\lambda, |x-sz_1|) \frac{(-z_1) \cdot (x-sz_1)}{|x-sz_1|}\, d\lambda \Big|
\les |t|^{-\frac{n}{2}} |x-sz_1| |z_1|
\end{equation*}
for each $s$.  If $s \in [0,1]$ we also have $|x-sz_1| \leq |x| + |z_1| \leq \la x\ra \la z_1 \ra$.
It follows that $\int_0^\infty e^{it\lambda^2} \lambda^{-1} \chi(\lambda) 
(R_0^+(\lambda^2) - R_0^-(\lambda^2) - K(\lambda,|x|)1)VP_e\,d\lambda$ maps $L^1$ to $L^{\infty,-1}$
provided $V$ has enough decay so that the range of $VP_e$ belongs to $L^{1,2}$, which follows from the fact that
$P_e:L^1\to L^\infty$, see Corollary~\ref{Pemapping}, and the decay of $V$.

Now if in addition $PVx = 0$ we can gain more by going to the second order expression
\begin{multline*}
K(\lambda, |x-z_1|) - K(\lambda, |x|) + \partial_rK(\lambda,|x|) \frac{z_1 \cdot x}{|x|}  \\
= \int_0^1 (1-s) \bigg[\partial_r^2 K(\lambda, |x-sz_1|)  \frac{(z_1 \cdot (x-sz_1))^2}{|x-sz_1|^2}\\
+ \partial_r K(\lambda, |x-sz_1|)\Big(\frac{|z_1|^2}{|x-sz_1|} - \frac{(z_1 \cdot (x-sz_1))^2}{|x-sz_1|^3}\Big)\bigg]\, ds.
\end{multline*}
Thanks to the bounds in~\eqref{eq:Omega1-2} and Lemma~\ref{lem:off_center_statphase} with $m=n$,
there is a uniform estimate
\begin{equation*}
\Big|\int_0^\infty e^{it\lambda^2}\lambda^{-1}\chi(\lambda) \partial_r^2K(\lambda, |x-sz_1|)
\frac{(z_1 \cdot (x-sz_1))^2}{|x-sz_1|^2}\, d\lambda\Big| \les 
|t|^{-\f n2} \la z_1\ra^2,
\end{equation*}
and similarly for each of the terms with $\partial_r K (\lambda, |x-sz_1|)$ using \eqref{eq:Omega1-2} repeatedly.  Plugging this back 
into the original operator integral yields
\begin{multline*}
\bigg\Vert
\int_0^\infty e^{it\lambda^2} \lambda^{-1}\chi(\lambda)
\big(R_0^+(\lambda^2) - R_0^-(\lambda^2) - K(\lambda, |x|)1 + \partial_rK(\lambda,|x|){\textstyle \frac{x}{|x|}}\cdot z_1\big)VP_e\, d\lambda
\bigg\Vert_{L^1\to L^\infty}\\ \les |t|^{-\frac{n}{2}},
\end{multline*}
provided $VP_e$ has range in $L^{1, 2}$, which is ensured
by Corollary~\ref{Pemapping} and the decay of $V$.

\end{proof}

\begin{corollary}\label{cor:2canc}

	The operator
	$$
		P_eV
		\frac{R_0^+(\lambda^2)-R_0^-(\lambda^2)}{\lambda^2}
	$$
	contributes $|t|^{1-\f n2}$ to \eqref{Stone} as an operator from $L^1$ to $L^\infty$.  If $P_eV1=0$, it
	contributes $|t|^{-\f n2}$ to \eqref{Stone} as an 
	operator from $L^{1,1}$ to $L^{\infty}$.  If in
	addition $P_eVx=0$ the contribution is as an operator
	from $L^1$ to $L^\infty$.

\end{corollary}

We are now ready to prove Theorem~\ref{thm:main}.

\begin{proof}[Proof of Theorem~\ref{thm:main}]

	We note that the Theorem is proven by bounding the
	oscillatory integral in the Stone formula
	\eqref{Stone},
\begin{align}
	\bigg|
	\int_0^\infty e^{it\lambda^2} \lambda \chi(\lambda)
	[R_V^+(\lambda^2)-R_V^-(\lambda^2)](x,y)\, 
	d\lambda\bigg| \les_{x,y} |t|^{-\alpha}
\end{align}	

We begin by proving Part~(\ref{thmpart1}), where there is
no $x,y$ dependence.  The proof follows by 
expanding $R_V^\pm(\lambda^2)$ into the Born series
expansion, \eqref{eq:finitebs} and \eqref{eq:bstail}.
The contribution of \eqref{eq:finitebs} is bounded by
$|t|^{-\f n2}$ by Proposition~\ref{bsprop}, while
the contribution of \eqref{eq:bstail} is bounded
by $|t|^{2-\f n2}P_eV1VP_e+O(|t|^{1-\f n2})$ by
Lemma~\ref{lem:nocanc}.

To prove Part~(\ref{thmpart2}), one uses 
Lemma~\ref{lem:1canc} in the place of Lemma~\ref{lem:nocanc}
in the proof of Part~(\ref{thmpart1}).  Finally,
Part~(\ref{thmpart3}) is proven by using
Lemma~\ref{lem:3canc}.

\end{proof}

We note that the proof of Theorem~\ref{thm:reg} is
actually simpler.  If zero is regular, the expansion 
of $M^{\pm}(\lambda)^{-1}$ is of the same form
with respect to the spectral variable $\lambda$ as
$(M^{\pm}(\lambda)+S_1)^{-1}$ given in 
Lemma~\ref{M+S inverse} with different operators that 
are still absolutely bounded and real-valued, see
Remark~\ref{rmk:reg}.
The dispersive bounds follow as in the analysis when
zero is not regular without the most singular terms that
arise from $-D_1/\lambda^2$.

We note that we need one further estimate on the
operator
$$
	\frac{R_0^+(\lambda^2)-R_0^-(\lambda^2)}{\lambda^2}
	VP_e
$$
that is not contained in Lemma~\ref{lem:2canc} to prove
the Corollary~\ref{cor:ugly} in the case that $P_eV1\neq 0$.  To establish that the operator with the $|t|^{1-\f n2}$ decay rate is indeed finite rank, and to see
why the operator $A_0(t)$ must map $L^{1,2}$ to 
$L^{\infty, -2}$ if $P_eV1\neq 0$, we need the following lemma 

\begin{lemma}\label{lem:L12}

	The operator
	$$
		\frac{R_0^+(\lambda^2)-R_0^-(\lambda^2)}{\lambda^2}
		VP_e
	$$
	contributes $c|t|^{1-\f n2}1VP_e+ O(|t|^{-\f n2})$ to \eqref{Stone}, where the error term is an operator from $L^1$ to $L^{\infty,-2}$. 

\end{lemma}

\begin{proof}

The desired bound 
follows using  \eqref{Jn canc} as in Lemma~\ref{lem:BS1}.  
We first concern ourselves with when 
$\lambda|x-z_1|\ll 1$,
in this case we note that
using \eqref{R0diffeps} out to one further term, we have
\begin{align}
		[R_0^+(\lambda^2)-R_0^-(\lambda^2)](x,z_1)&=\lambda^{n-2} G_{n-2}^c+\lambda^n G_n^c
		+\widetilde O(\lambda^{n-2} (\lambda|x-z_1|)^{2+\epsilon}),
		\, 0\leq \epsilon <2.
\end{align}
Recalling that $G_n^c(x,z_1)=c_n |x-z_1|^2$,
we can now write (for $\lambda|x-z_1|\ll 1$)
\begin{multline*}
	\frac{R_0^+-R_0^-(\lambda^2)(x,z_1)}{\lambda^2} 
	V(z_1) P_e(z_1,y)
	= c_{n-2}\lambda^{n-4}V(z_1) P_e(z_1,y)
	+\lambda^{n-2}  |x-z_1|^2 V(z_1) P_e(z_1,y)\\
	+ \widetilde O(\lambda^{n-4} (\lambda|x-z_1|)^{2+\epsilon})V(z_1) P_e(z_1,y).
\end{multline*}
The first $\lambda^{n-4}$ term can be seen to contribute
$c|t|^{1-\f n2}$ to \eqref{Stone} by Lemma~\ref{lem:IBP}.
Similarly the second term with $\lambda^{n-2}$ is seen
to contribute $\la x\ra^2 |t|^{-\f n2}$ to \eqref{Stone}
by Lemma~\ref{lem:IBP}.
The final error term is controlled
identically to how one bounds \eqref{eq:Bornlowbad}
in Lemma~\ref{lem:BS1} (with an additional factor of
$|x-z_1|^2$), from which one again has a
contribution of size 
$\la x\ra ^2|t|^{-\f n2}$ to \eqref{Stone}.

On the other hand, if $\lambda|x-z_1|\gtrsim 1$, we
can write
$$
	[R_0^+-R_0^-](\lambda^2)(x,z_1)
	=e^{i\lambda |x-z_1|}\widetilde O(
	\lambda^{n-2}(\lambda|x-z_1|^{\f 12+\alpha}))+
	e^{-i\lambda |x-z_1|}\widetilde O(
	\lambda^{n-2}(\lambda|x-z_1|^{\f 12+\alpha})).
$$
As usual, the most delicate term is the `-' phase.
We need to control the contribution of
\begin{align*}
	\int_0^\infty e^{it\lambda^2}\lambda^{-1} \chi(\lambda) 	e^{-i\lambda |x-z_1|}\widetilde O(
	\lambda^{n-2}(\lambda|x-z_1|)^{\f 12+\alpha})\, d\lambda.
\end{align*}
Upon integrating by parts $\f n2-1$ times against
the imaginary Gaussian, we are
left to bound an integral of the form
\begin{align*}
	|t|^{1-\f n2}\int_0^\infty 
	e^{it\lambda^2-i\lambda|x-z_1|} a(\lambda) \, d\lambda
\end{align*}
where 
\begin{align*}
	|a(\lambda)| \les \lambda^{-1} (\lambda|x-z_1|)^{\f 12+\alpha}\les \lambda^{\f 12} |x-z_1|^{\f 32}
	\les |x-z_1|^2 \bigg(\frac{ \lambda^{\f 12}}
	{|x-z_1|^{\f 12}}
	\bigg)
\end{align*}
where we took $\alpha=1$ in the second to last line.
Similarly,
\begin{align*}
	|a'(\lambda)|\les |x-z_1|^2 \bigg(\frac{1}{\lambda^{\f 12} |x-z_1|^{\f 12}}	\bigg).
\end{align*}
Now, one can employ Lemma~\ref{stat phase} as in the
proof of Lemma~\ref{lem:BS1} (with an extra factor of
$|x-z_1|^2$) to see that this term contributes at
most $\la x\ra^2 |t|^{-\f n2}$ to \eqref{Stone}.
The `+' phase again follows more simply from another 
integration by parts, this time against
$e^{it\lambda^2+i\lambda |x-z_1|}$.

\end{proof}

\begin{corollary}\label{cor:L12}

	The operator
	$$
		P_e V\frac{R_0^+(\lambda^2)-R_0^-(\lambda^2)}
		{\lambda^2}
	$$
	contributes $c|t|^{1-\f n2}P_eV1+ O(|t|^{-\f n2})$ to \eqref{Stone}, where the error term is an operator from $L^{1,2}$ to $L^{\infty}$. 

\end{corollary}

The proof of the corollary is identical in form to the
proof of Lemma~\ref{lem:L12} with the spatial variables
$x$ and $y$ trading places.


\section{Spectral characterization and integral estimates}\label{sec:Spec}

We provide a characterization of the spectral subspaces of
$L^2(\R^n)$ that are related to the invertibility of
certain operators in our expansions.  This characterization
and its proofs are identical to those given in \cite{GGodd}, as such 
we provide the statements and omit the proofs.  As in the
odd case, the lack of resonances in dimensions $n>4$
simplifies these characterizations.
In addition, we state several oscillatory 
integral estimates from
\cite{GGodd} and provide proofs for new integral estimates
that are required in this paper.

\begin{lemma}\label{S characterization}

	Assume that $|V(x)|\les \la x\ra^{-2\beta}$ for some $\beta\geq 2$,
	$f\in S_1L^2(\R^n)\setminus\{0\}$ for
	$n\geq 5$ iff $f=wg$ for $g\in L^2\setminus\{0\}$ such that
	$-\Delta g+Vg=0$ in $\mathcal S'$.

\end{lemma}

\begin{lemma}\label{D1 lemma}
 
  The kernel of $S_1vG_1^0vS_1$ is trivial in $S_1L^2(\R^n)$
  for $n\geq 5$.

\end{lemma}

We note that the proof in the odd dimensional case 
involves the operator $G_2$ in place of the operator
$G_1^0$.  This is a notational discrepancy only, both
of these operators have integral kernel which is a 
scalar multiple of $|x-y|^{4-n}$.

 


\begin{lemma}\label{lem:eproj}
 
    The projection onto the eigenspace at zero is $G_0^0vS_1[S_1vG_1^0vS_1]^{-1}S_1vG_0^0$.
    That is,
    \begin{align}\label{Pe defn}
    	P_e=G_0^0vD_1vG_0^0	.
    \end{align}

\end{lemma}

\begin{lemma}\label{lem:efnLinf}

	Assume that $|V(x)|\les \la x\ra^{-\beta}$ for some $\beta>2$,	
	If $g\in L^2$ is a solution of $(-\Delta+V)g=0$
	then $g\in L^\infty$.

\end{lemma}

\begin{corollary}\label{Pemapping}

	$P_e$ is bounded operator from $L^1$ to $L^\infty$.

\end{corollary}

In addition we have the following oscillatory integral
bounds which prove useful in the preceding analysis.
Some of these Lemmas along with their proofs appear 
in Section~6 of \cite{GGodd},
accordingly we state them without proof.

\begin{lemma}\label{lem:IBP}

	If $k\in \mathbb N_0$, we have the bound
	$$
		\bigg|\int_0^\infty e^{it\lambda^2} \chi(\lambda)
		\lambda^k \, d\lambda\bigg| \les |t|^{-\frac{k+1}{2}}.
	$$

\end{lemma}

\begin{lemma}\label{lem:fauxIBP}

	For a fixed $\alpha>-1$,
	let $f(\lambda)=\widetilde O_{k+1}(\lambda^\alpha)$ be supported on the 
	interval $[0,\lambda_1]$ for some $0<\lambda_1\les 1$.
	Then, if $k$ satisfies $-1<\alpha-2k< 1$ we have
	\begin{align*}
		\bigg|\int_0^\infty e^{it\lambda^2} 
		f(\lambda)\, d\lambda\bigg|&\les |t|^{-\frac{\alpha+1}{2}}.
	\end{align*}
	
\end{lemma}

The following two bounds take advantage of the fact that
$n$ is even and hence $\f n2$ is an integer.

\begin{lemma}\label{lem:fauxIBP2}

	If $\alpha>n-3$ and
	$f(\lambda)=\widetilde O_{\f n2-1}(\lambda^{\alpha})$  supported on the 
	interval $[0,\lambda_1]$ for some $0<\lambda_1\les 1$.
	Then, 
	\begin{align*}
		\bigg|\int_0^\infty e^{it\lambda^2} 
		f(\lambda)\, d\lambda\bigg|&\les |t|^{1-\frac{n}{2}}.
	\end{align*}
	
\end{lemma}

\begin{proof}

	The powers of $\lambda$ allow us to integrate by
	parts $\f n2-1$ times with no boundary terms,
	we are left to bound
	$$
		|t|^{1-\f n2}\int_0^\infty e^{it\lambda^2}
		\widetilde O(\lambda^{-1+})\, d\lambda.
	$$
	By the assumption that the integral is supported
	on $[0,\lambda_1]$ the integral is bounded.

\end{proof}

\begin{corollary}\label{cor:fauxIBP2}

	If $\alpha>n-1$ and
	$f(\lambda)=\widetilde O_{\f n2}(\lambda^{\alpha})$  supported on the 
	interval $[0,\lambda_1]$ for some $0<\lambda_1\les 1$.
	Then, 
	\begin{align*}
		\bigg|\int_0^\infty e^{it\lambda^2} 
		f(\lambda)\, d\lambda\bigg|&\les |t|^{-\frac{n}{2}}.
	\end{align*}
	
\end{corollary}

The following proof completes the
dispersive bounds proven in Section~\ref{sec:sing}.

\begin{proof}[Proof of Lemma~\ref{lem:off_center_statphase}]
Assume that $t > 0$.  The proof for $t<0$ is identical with the $\pm$ signs reversed.
Suppose the phase angle $e^{i\lambda r}$ carries a positive sign.  In this case there is no
stationary phase point of $e^{it\lambda^2 + \lambda r}$ in the domain of intergation.
One can estimate trivially that
\begin{equation*}
\Big|\int_0^{t^{-1/2}} e^{i(t\lambda^2+\lambda r)}\lambda^{m-1}
\Omega(\lambda r)\chi(\lambda)\,d\lambda\Big|
\les t^{-\frac{m}{2}},
\end{equation*}
and repeated integration by parts against
$e^{it(\lambda^2+\lambda \frac{r}{t})}$
($\frac{m}{2}$ times if $m$ is even,
$\frac{m+1}{2}$ if $m$ is odd) gives the result.  It is convenient to note that
$|(\frac{d}{d\lambda})^k \Omega(\lambda r)| \les \max(r,\lambda^{-1})^k \la \lambda r \ra^{\frac{1-m}{2}}$,
so differentiating this expression has a similar effect as when derivatives act on
the monomial $\lambda^{m-1}$ and is better behaved when $\lambda r$ is small.

All boundary terms of the repeated integration by parts can be controlled using the crude bound
$|\lambda + \frac{r}{2t}| \geq |\lambda|$.  Most of the integral terms are controlled this way as well,
but if $m$ is even this creates a few apparent terms of the form
$\int_{t^{-1/2}}^\infty \big| \lambda^{-1} \Omega(\lambda r)\chi(\lambda)\big|\, d\lambda$
if all derivatives fall on powers of $\lambda$ or $(\lambda + \frac{r}{2t})$.  In fact no such
terms occur, due to cancellation in the derivative $\frac{d}{d\lambda}\big(\frac{\lambda}{\lambda+ \frac{r}{2t}}\big)
= \frac{r}{2t\,(\lambda + r/2t)^2}$.  That leads instead to integrals of the form
\begin{align*}
\frac{r}{2t}\int_{t^{-1/2}}^\infty \big| (\lambda+r/2t)^{-2}\Omega(\lambda r)\chi(\lambda)\big|\,d\lambda \les \frac{r}{2t}
\int_{0}^\infty \frac{1}{(\lambda+r/2t)^2}\, d\lambda\les 1.
\end{align*}

Now consider the phase angle $e^{-i\lambda r}$, which causes
$e^{i(t\lambda^2 - \lambda r)}$ to have  a stationary point $\lambda_0=\frac{r}{2t}$.
If  $r <4 \sqrt{t}$, then $0\leq\lambda_0<2t^{-\f 12}$, and the integral
can be estimated in the same manner as above, splitting the domain into the
two pieces $(0,4t^{-\frac12})$ and $(4t^{-\f12},\infty)$.  On the
first interval, the bound is clear.   On the second
interval, the comparison $|\lambda - \lambda_0| \approx |\lambda|$
controls all boundary terms and most of the integral terms as before.  For the
exceptional integrals, the last bound comes from estimating
\begin{align*}
\frac{r}{2t}\int_{4t^{-1/2}}^\infty \big| (\lambda-r/2t)^{-2}\Omega(\lambda r)\chi(\lambda)\big|\,d\lambda \les \frac{r}{2t}
\int_{\frac{r}{t}}^\infty \frac{1}{(\lambda-r/2t)^2}\,d\lambda \les 1.
\end{align*}

If  $r > 4\sqrt{t}$, then $\lambda_0 > 2t^{-\frac12}$.
Here we apply stationary phase estimates to the interval $(\lambda_0-t^{-\f 12},\lambda_0+t^{-\f12})$.  On this interval
one can approximate $\lambda \approx \lambda_0$, and consequently
$|\lambda^{m-1} \Omega(\lambda r)|\approx |\lambda_0^{m-1} \Omega(\lambda_0 r)| \les t^{\frac{1-m}{2}}$.
So this integral over the interval $|\lambda - \lambda_0| < t^{-1/2}$  contributes no more
than $t^{-m/2}$ as desired.

Noting that $\partial_\lambda e^{it(\lambda-\lambda_0)^2}
=2it(\lambda-\lambda_0)e^{it(\lambda-\lambda_0)^2}$,
integration by parts 
on the interval $[\lambda_0 + t^{-1/2}, +\infty)$ is relatively straightforward.
Since $\lambda > \lambda - \lambda_0>t^{-\f 12}$, the worst behavior occurs when all derivatives act on powers
of $(\lambda - \lambda_0)$.  For all boundary terms arising in this manner it suffices to observe that
$\lambda - \lambda_0 = t^{-1/2}$ and
$|\lambda^{m-1}\Omega(\lambda r)| \les t^{\frac{1-m}{2}}$ at the left endpoint.  The integral terms
is controlled by the estimate
\begin{align*}
	t^{-k}\int_{\lambda_0 + t^{-1/2}}^\infty \frac{\lambda^{m-1}\Omega(\lambda r)}{(\lambda-\lambda_0)^{2k}}
	\,d\lambda
	&\les t^{-k}\int_{\lambda_0 + t^{-1/2}}^{2\lambda_0} 
	\frac{|\lambda_0^{m-1}\Omega(\lambda_0 r)|}{(\lambda-\lambda_0)^{2k}} \,d\lambda 
	+ t^{-k} \int_{2\lambda_0}^\infty \frac{|\Omega(\lambda r)|}{\lambda^{2k+1-m}} \,d\lambda
\end{align*}
We note that we still have 
$|\lambda_0^{m-1} \Omega(\lambda_0 r)| \les t^{\frac{1-m}{2}}$, thus by a simple change of variables
we can bound the first integral by
\begin{align*}
	t^{\frac{1-m}{2}-k}\int_{t^{-\f12}}^\infty s^{-2k}\,
	ds \les t^{-\f m2},
\end{align*}
provided $2k>1$.  For the second integral, we have that
$|\Omega(\lambda r)|\les (\lambda r)^{\frac{1-m}{2}}$,
and $2\lambda_0=r/t$, 
so we need to bound
\begin{align*}
	r^{\frac{1-m}{2}}t^{-k} \int_{r/t}^\infty
	\lambda^{\frac{m-1}{2}-2k}\, d\lambda
	&\les r^{\frac{1-m}{2}}t^{-k}
	\bigg(\frac{r}{t}
	\bigg)^{\frac{m+1}{2}-2k} \les r^{1-2k}t^{k-\frac{m+1}{2}} \les t^{-\f m2}
\end{align*}
provided $2k > \max(1,\frac{m+1}{2})$.  
Here we used that $r > 4\sqrt{t}$ in the last inequality.

Integration by parts on the interval $[0, \lambda_0 - t^{-1/2})$ is only slightly more complicated.
For all $m > 2$ there are no boundary terms at $\lambda= 0$, and if $m = 1$ the boundary term
has size $(\lambda_0 t)^{-1} \approx r^{-1} \les t^{-\f 12}$ since $r > 4\sqrt{t}$. 
The boundary terms at $\lambda_0 - t^{-\f 12}$ are handled identically to the ones at
$\lambda_0 + t^{-\f 12}$ in the previous case.

When $m$ is even, after integrating by parts $\f m2$ times,
the main integral  consists of expressions with the form
\begin{equation}\label{eq:Omegaosc}
t^{-\frac{m}{2}} \int_0^{\lambda_0 - t^{-\frac12}} 
\big|\lambda^{m-1-j} (\lambda - \lambda_0)^{j+\ell-m}
 r^\ell \Omega^{(\ell)}(\lambda r)|\, d\lambda
\end{equation}
with $j + \ell \leq \frac{m}{2}$.
There are three regimes to consider: $\lambda \in (0,\frac{1}{r})$, 
$\lambda \in (\frac{1}{r}, \frac{\lambda_0}{2})$, and $\lambda \in (\frac{\lambda_0}{2}, \lambda_0- t^{-\frac12})$.
In the first regime we use that $|\Omega^{(\ell)}(\lambda r)|\les 1$ and $|\lambda-\lambda_0|\approx \lambda_0$,
to see that this integral contributes at most
$t^{-\frac{m}{2}}(\frac{t}{r^2})^{m-j-\ell}\les t^{-\f m2}$ to the \eqref{eq:Omegaosc}. 
On the second regime, we again have $|\lambda-\lambda_0|\approx \lambda_0$
but now $|\Omega^{(\ell)}(\lambda r)|\les (\lambda r)^{\frac{1-m}{2}-\ell}$.  The contribution of this
regime to the integral is now bounded by
\begin{align*}
	t^{-\f m2}\lambda_0^{j+\ell-m} r^{\frac{1-m}{2}}
	\int_0^{\lambda_0} \lambda^{\frac{m-1}{2}-j-\ell}
	\, d\lambda \les t^{-\f m2} r^{\frac{1-m}{2}} 
	\lambda_0^{\frac{1-m}{2}} =t^{-\f m2} \bigg(\frac{\sqrt{t}}{r}\bigg)^{(m-1)}\les
	t^{-\f m2}.
\end{align*}
Since $\frac{m-1}{2}-j-\ell>-1$, we safely extended
the lower limit of integration to zero.

On the last regime we note that $\lambda \approx \lambda_0$, so that if we use $s=\lambda_0-\lambda$ we
can bound the contribution by
\begin{align*}
	t^{-\f m2} \lambda_0^{m-1-j} \Omega^{(\ell)}(\lambda_0 r) r^{\ell} \int_{t^{-\f 12}}^{\lambda_0} s^{j+\ell-m}
	\, ds.
\end{align*}
We first consider the case in which
$j+\ell-m <-1$, then we can bound this integral by
\begin{align*}
	t^{-\f m2} r^{\frac{1-m}{2}} \lambda_0^{\frac{m-1}{2}-j-\ell} \int_{t^{-\f 12}}^\infty  s^{j+\ell-m}\, ds &\les
	t^{-\f m2} \bigg( \frac{t^{\frac{1}{2}}}{r}\bigg)^{j+\ell}
	\les t^{-\f m2}.
\end{align*}
The one exception is if $m=2$ and $j + \ell = 1$, then 
we cannot extend the region of integration off to infinity,
but instead note that
$$
 	\int_{t^{-\f 12}}^{\lambda_0/2}  s^{j+\ell-m}\, ds
 	=\int_{t^{-\f 12}}^{\lambda_0/2}   s^{-1}\, ds
 	=\log \bigg(\frac{\lambda_0}{2t^{-\f 12}}
 	\bigg)
$$
So that in this case the third
region instead contributes $t^{-\f m2}(\frac{\sqrt{t}}{r})\bigl|\log(\frac{4\sqrt{t}}{r})\bigr|$
which is still uniformly bounded by $t^{-\f m2}$ since $\sqrt{t}/r<\frac14$.

When $m$ is odd the representative expressions are
\begin{equation*}
t^{-\frac{m+1}{2}} \int_0^{\lambda_0 - t^{-\frac12}} 
\big|\lambda^{m-1-j} (\lambda - \lambda_0)^{j+\ell-m-1}
 r^\ell \Omega^{(\ell)}(\lambda r)\bigr|\, d\lambda
\end{equation*}
with $j + \ell \leq \frac{m+1}{2}$.  After breaking the integral into the same three
regimes, one can similarly show that the contribution of each one is bounded by $t^{-\frac{m}{2}}$
as above.  There is again a logarithmic issue in the second regime if $j+\ell = \frac{m+1}{2}$
and in third regime if $m=1$ and $j+\ell=1$.  Both are resolved by the fact that
$(\frac{\sqrt{t}}{r})^m\bigl|\log(\frac{4\sqrt{t}}{r})\bigr|$ is uniformly bounded over $r > 4\sqrt{t}$.

\end{proof}

Finally we note the non-oscillatory integral estimate
which is proven in \cite{EG}.

\begin{lemma}\label{EG:Lem}

	Fix $u_1,u_2\in\R^n$ and let $0\leq k,\ell<n$, 
	$\beta>0$, $k+\ell+\beta\geq n$, $k+\ell\neq n$.
	We have
	$$
		\int_{\R^n} \frac{\la z\ra^{-\beta-}}
		{|z-u_1|^k|z-u_2|^\ell}\, dz
		\les \left\{\begin{array}{ll}
		(\frac{1}{|u_1-u_2|})^{\max(0,k+\ell-n)}
		& |u_1-u_2|\leq 1\\
		\big(\frac{1}{|u_1-u_2|}\big)^{\min(k,\ell,
		k+\ell+\beta-n)} & |u_1-u_2|>1
		\end{array}
		\right.
	$$
	Furthermore,
	$$
		\int_{\R^n} \frac{\la z\ra^{-\beta-}}
		{|z-u_1|^k|z-u_2|^\ell}\, dz
		\les 
		\bigg(\frac{1}{|u_1-u_2|}\bigg)^{\alpha},
	$$
	where one can take $\alpha=\max(0,k+\ell-n)$ or
	$\alpha=\min(k,\ell,
			k+\ell+\beta-n)$.

\end{lemma}

\end{document}